\documentclass[11pt]{amsart}
\usepackage[
    colorlinks=true,       
    linkcolor=blue,          
    citecolor=blue,        
    filecolor=blue,      
    urlcolor=blue]{hyperref}

\usepackage{stmaryrd,mathrsfs}
\usepackage{xcolor}
\numberwithin{equation}{section}
\allowdisplaybreaks[4]
\usepackage{amsmath,amsthm,amscd,amssymb,amsfonts, amsbsy}
\usepackage{latexsym}
\usepackage{exscale}
\newtheorem{Theorem}[equation]{Theorem}
\newtheorem{Proposition}[equation]{Proposition}
\newtheorem{Corollary}[equation]{Corollary}
\newtheorem{Lemma}[equation]{Lemma}

\def\Xint#1{\mathchoice
{\XXint\displaystyle\textstyle{#1}}%
{\XXint\textstyle\scriptstyle{#1}}%
{\XXint\scriptstyle\scriptscriptstyle{#1}}%
{\XXint\scriptscriptstyle\scriptscriptstyle{#1}}%
\!\int}
\def\XXint#1#2#3{{\setbox0=\hbox{$#1{#2#3}{\int}$}
\vcenter{\hbox{$#2#3$}}\kern-.5\wd0}}

\def\dashint{\Xint-}

\def\supp{{\rm supp}\ }

\def\bbR{\mathbb{R}}
\def\Rn{\mathbb{R}^n}

\def\ch{\rm{ch}}
\def\dw{\textup{d}w}

\newcommand{\abs}[1]{|#1|}
\newcommand{\ave}[1]{\langle #1 \rangle}

\def\re{\textup{Re}}

\begin{document}

\title[New bounds for bilinear CZO]{New bounds for bilinear  Calder\'on-Zygmund \\ operators and applications}

\author[Wendol\'in Dami\'an, Mahdi Hormozi and Kangwei Li]{Wendol\'in Dami\'an, Mahdi Hormozi and Kangwei Li}

\address[W. Dami\'an]{Department of Mathematics and Statistics, P.O.B.~68 (Gustaf H\"all\-str\"omin katu~2b), FI-00014 University of Helsinki, Finland}

\email{wendolin.damian.gonzalez@gmail.com}

\address[K. Li]{Department of Mathematics and Statistics, P.O.B.~68 (Gustaf H\"all\-str\"omin katu~2b), FI-00014 University of Helsinki, Finland. }
\curraddr{\sc{BCAM--Basque Center for Applied Mathematics, Mazarredo, 14. 48009 Bilbao Basque Country, Spain}}

\email{kli@bcamath.org}

\address[M. Hormozi]{Department of Mathematical Sciences, Division of Mathematics,
University of Gothenburg,
Gothenburg 41296, Sweden \& 
 Department of Mathematics, Shiraz University, Shiraz 71454, Iran}

\email{me.hormozi@gmail.com}

\thanks{W.D. and K.L. were supported by the European Union through T. Hyt\"onen's ERC
Starting Grant
``Analytic-probabilistic methods for borderline singular integrals''.
They are members of the Finnish Centre of Excellence in Analysis and
Dynamics Research. W.D. is also supported by the Spanish Ministry of Economy and Competitiveness through the project MTM2015-63699-P ``An\'alisis Complejo, Espacios de Banach y Operadores''. K.L. is also supported by the Basque Government through the BERC 2014-2017 program and by Spanish Ministry of Economy and Competitiveness MINECO: BCAM Severo Ochoa
excellence accreditation SEV-2013-0323.
}

\date{\today}


\begin{abstract}
In this work we extend Lacey's domination theorem to prove the pointwise control of bilinear Calder\'on--Zygmund operators with Dini--continuous kernel by sparse operators. The precise bounds are carefully tracked following the spirit in a recent work of Hyt\"onen, Roncal and Tapiola. We also derive new mixed weighted estimates for a general class of bilinear dyadic positive operators using multiple $A_{\infty}$ constants inspired in the Fujii-Wilson and Hrus\v{c}\v{e}v classical constants. These estimates have many new applications including mixed bounds for multilinear Calder\'on--Zygmund operators and their commutators with $BMO$ functions, square functions and multilinear Fourier multipliers.
\end{abstract}
\maketitle

\section{Introduction}\label{Sect:1}

In the last decades, several advances have been carried out in the fruitful area of weighted inequalities concerning the precise determination of the optimal bounds of the weighted operator norm of Calder\'on--Zygmund operators in terms of the $A_p$ constant of the weights. It has been a long journey from the proof of the linear dependence on the $A_2$ constant of $w$ of the $L^2(w)$ norm of the Ahlfors-Beurling transform \cite{PV} leading to the full proof of the $A_2$ theorem due to T. Hyt\"onen \cite{Hytonen_Annals}, plenty of previous partial attempts by others. We refer the interested reader to \cite{Hytonen_Annals,Ler1} and the references therein for a survey on the advances on the topic.

After Hyt\"onen's proof, A. Lerner \cite{Ler1} gave an alternative proof of the $A_2$ theorem which showed that Calder\'on--Zygmund operators can be controlled in norm from above by a very special dyadic type operators defined by means of the concept of \textit{sparseness}. More precisely, if $\mathcal{S}$ is a collection of dyadic cubes within a dyadic grid $\mathcal{D}$ (see Sect.~\ref{Sect:2} for the definition), we say that the operator $\mathcal{A}_{\mathcal{S}}$ is \textit{sparse} if
\begin{equation}
\mathcal{A}_{\mathcal{S},\mathcal{D}} f(x) = \sum_{Q\in\mathcal{S}} \langle f\rangle_Q1_Q(x),
\end{equation}
where $\mathbf 1_Q$ is the characteristic function of the cube $Q$ and the collection $\mathcal{S}$ satisfies that there exists some $\gamma\in (0,1)$ such that for each $Q\in\mathcal{S}$,
$$\sum_{S'\in ch_{\mathcal{S}}(S)}|S'| \leq \gamma |S|,$$
for every $S\in\mathcal{S}$. Here $ch_{\mathcal{S}}(S)$ denotes the set of the $\mathcal{S}$-children of a dyadic cube $S$. Namely, the set of the maximal cubes $S'\in\mathcal{D}$ such that $S'\subsetneq S$.
One remarkable aspect from Lerner's proof is its flexibility to be adapted to the multilinear setting. In fact, in \cite{DLP} the first author, A. Lerner and C. P\'erez proved that multilinear Calder\'on--Zygmund operators can be controlled from above in norm by a supremum of sparse operators. More precisely, if $X$ is a Banach function space over $\Rn$ equipped with the Lebesgue measure, it holds that for any appropriate $\vec{f}$,
\begin{equation}\label{eq:normcontrolCZO}
  ||T(\vec{f})||_{X} \leq C_{T,n,m} \sup_{\mathcal{D},\mathcal{S}} ||\mathcal{A}_{\mathcal{D},\mathcal{S}}(|\vec{f}|)||_{X},
\end{equation}
where
\begin{equation}
\mathcal{A}_{\mathcal{D},\mathcal{S}}(\vec{f}):= \sum_{Q\in\mathcal{S}}\prod_{i=1}^m \langle f_i \rangle 1_Q,
\end{equation}
and the supremum is taken over arbitrary dyadic grids $\mathcal{D}$ and sparse families $\mathcal{S}\in\mathcal{D}$.  As an application of this result, it was derived a multilinear analogue of the $A_2$ theorem, proving that in this more general scenario, a linear bound on the corresponding multiple weight constant also holds. Lately, this result was extended by the third author, K. Moen and W. Sun \cite{LMS} who proved the sharp bounds for the class of multilinear sparse operators from which follows the sharp bounds for Calder\'on--Zygmund operators. More precisely, if $\overrightarrow{w}=(w_1,\ldots,w_m)$ are weights, $1<p_1,\ldots,p_m<\infty$ and $p$ are numbers verifying that $\frac{1}{p}=\frac{1}{p_1}+\ldots+\frac{1}{p_m}$ and we denote $\overrightarrow P = (p_1,\ldots,p_m)$,
\begin{equation}
  ||\mathcal{A}_{\mathcal{S},\mathcal{D}}(\vec{f})||_{L^p(\nu_{\vec{w}})} \lesssim [\vec{w}]_{A_{\vec{P}}}^{\max{(1,\frac{p_1'}{p},\ldots,\frac{p_m'}{p})}}\prod_{i=1}^m ||f_i||_{L^{p_i}(w_i)}.
\end{equation}
Here $\nu_{\vec{w}}=\prod_{i=1}^m w_i^{p/p_i}$ and the multiple $A_{\vec{P}}$ constant is defined as follows,
\begin{equation}\label{eq:AP_multiple}
[\vec{w}]_{A_{\vec{P}}}=\sup_{Q}\left(\frac{1}{|Q|}\int_Q\nu_{\vec{w}}\right)\prod_{i=1}^m\left(\frac{1}{|Q|}\int_Q w_i^{1-p'_i}\right)^{p/p'_i}<\infty.
\end{equation}
However, the problem of finding the sharp bounds in the multilinear setting for the full range of exponents was still open, since \eqref{eq:normcontrolCZO} does not apply if $1/m<p<1$, in which case $L^p(\nu_{\vec{w}})$ is not a Banach function space.

Later on, this problem was solved independently by A. Lerner and F. Nazarov \cite{LN} and J.M. Conde--Alonso and G. Rey \cite{C-AR}. The main idea in both works was a pointwise control of multilinear Calder\'on--Zygmund operators by sparse operators avoiding the use of the adjoint operators and duality, which was the key point in Lerner's original proof.

Another remarkable improvement in \cite{C-AR, LN} was considering weaker regularity conditions on the kernels of Calder\'on--Zygmund operators.  In fact, in both works it was considered the case of log-Dini continuous kernels.
Notwithstanding, this pointwise control also holds in the linear setting under the weaker Dini condition, as recently shown by M. Lacey \cite{Lac} in a qualitative way or, shortly after, by T. Hyt\"onen, L. Roncal and O. Tapiola \cite{HRT} tracking the precise dependence on the constants.

The aim of this note is two-fold. On one hand, we prove the pointwise control by sparse operators of bilinear Calder\'on--Zygmund operators with Dini-cotinuous kernels taking care of the precise constants.

On the other hand, we prove three different mixed bounds for a general class of bilinear dyadic positive operators using the parallel stopping cubes technique. The first bound (see Theorem~\ref{Thm:1}) follows the spirit in the work of the third author and W. Sun \cite{LS}, combining a product of the $A_{\vec{P}}$ and $A_{\infty}$ linear constants of the weights involved. The other two mixed weighted bounds combine the multiple $A_{\vec{P}}$ constant with natural extensions of the linear Hrus\v{c}\v{e}v and Fujii-Wilson $A_{\infty}$ constants (Theorems~\ref{Thm:2} and \ref{Thm:3}).

As a consequence, we are able to extend these weighted bounds to multilinear Calder\'on--Zygmund operators with Dini-continuous kernels and obtain new precise weighted bounds for their commutators with $BMO$ functions, square functions and Fourier multipliers in the multiple scenario.

For the sake of simplicity, throughout this paper we are mainly going to consider the bilinear case. Notwithstanding, a similar argument can be used to obtain the general multilinear case. Observe that in the section concerning commutators we give the general proof since it is more convenient.

The organization of this paper is as follows. In Section~\ref{Sect:2} we give some background and definitions which will be useful to prove our main results. In Section~\ref{Sect:4} we prove the pointwise control of multilinear Calder\'on--Zygmund operators by sparse operators whereas in Section~\ref{Sect:5} we obtain three quantitative bounds for a general class of positive dyadic operators. In Section~\ref{Sect:6}, main results in the previous section are applied to derive mixed weighted bounds for commutators of multilinear Calder\'on--Zygmund operators as well as for multilinear square functions and Fourier multipliers. Finally, in Section~\ref{Sect:Appendix} we prove quantitative versions of some classical boundedness results in the multilinear setting.

Throughout this paper, we will denote the average of a function $f$ over a cube $Q$ as
\begin{equation}
  \langle f \rangle_Q = \dashint_Q f = \dashint_Q f(x)dx = \frac{1}{|Q|}\int_Q f(x)dx,
\end{equation}
where $|Q|$ denotes the Lebesgue measure of $Q$. If $w$ is a weight, i.e. a measurable locally integrable function defined in $\Rn$ taking values in $(0,\infty)$ for almost every point, we will denote $w(Q):=\int_Q w(x)dx$ and $w \mathbf 1_Q(x):=w(x)\mathbf 1_Q(x)$. We will use the notation $A\lesssim B$ to indicate that there is a constant $c$, independent of the weight constant, such that $A\leq c B$.

\section{Preliminaries}\label{Sect:2}

\subsection{$\omega$-bilinear Calder\'on--Zygmund operators}
We say that $T$ is a $\omega$-bilinear Calder\'on--Zygmund operator if it is a bilinear operator originally defined on the product of Schwartz spaces and taking values into the space of tempered distributions,
\begin{equation}
  T:\mathcal{S}(\Rn)\times\mathcal{S}(\Rn)\to \mathcal{S}'(\Rn),
\end{equation}
and for some $1\leq q_1,q_2<\infty$ it extends to a bounded bilinear operator from $L^{q_1}\times L^{q_2}$ to $L^q$, where $1/q_1+1/q_2=1/q$, and if there exists a function $K$, defined off the diagonal $x=y=z$ in $(\Rn)^3$, satisfying
\begin{equation}
  T(f_1,f_2)(x)= \iint_{(\Rn)^2} K(x,y,z)f_1(y)f_2(z)dydz,
\end{equation}
for all $x\notin \supp{f_1}\cap\supp{f_2}$. The kernel $K$ must also satisfy, for some constants $C_K>0$ and $\tau\in(0,1)$, the following size condition
\begin{equation}\label{eq:size}
  |K(x,y,z)|\leq\frac{C_K}{(|x-y|+|x-z|)^{2n}},
\end{equation}
and, the smoothness estimate
\begin{align*}\label{eq:smoothness}
    |K(x&+h,y,z)-K(x,y,z)| +   |K(x,y+h,z)-K(x,y,z)| \\   &+|K(x,y,z+h)-K(x,y,z)| \\
    &\leq \frac{1}{(|x-y|+|x-z|)^{2n}} \omega\left( \frac{|h|}{|x-y|+|x-z|}\right),
\end{align*}
whenever $|h|\leq \tau\max{(|x-y|,|x-z|)}$.

If $\omega:[0,\infty)\to[0,\infty)$ is a modulus of continuity (i.e. it is increasing, subadditive ($\omega(t+s)\leq \omega(t)+\omega(s)$) and $\omega(0)=0$), the kernel $K$ is said to be a log-Dini-continuous kernel if $\omega$ satisfies the following condition
\begin{equation}\label{eq:logDini}
||\omega||_{\textup{log-Dini}}:=\int_0^1 \omega(t)\left(1+\log\left(\frac{1}{t}\right)\right)\frac{dt}{t} < \infty.
\end{equation}

We are mostly interested in the weaker case when $K$ is a $Dini(a)$-conti-nuous kernel. Namely, when $\omega$ satisfies the following condition:
\begin{equation}\label{eq:Dini}
  ||\omega||_{\textup{Dini(a)}}:= \int_0^1 \omega^a(t)\frac{dt}{t}<\infty.
\end{equation}

In the case $a=1$, we will denote $||\omega||_{\textup{Dini(a)}}$ simply as $||\omega||_{\textup{Dini}}$.

Given a bilinear Calder\'on-Zygmund operator $T$, the maximal truncation of $T$ is defined as the operator $T_{\sharp}$ given by
\begin{equation}\label{eq:MaxTru}
  T_{\sharp}(f_1,f_2)(x)=\sup_{\varepsilon>0} \left| T_{\varepsilon}(f_1,f_2)(x) \right|,
\end{equation}
where $T_{\varepsilon}$ is the $\varepsilon$-truncation of $T$
\begin{equation}
  T_{\varepsilon}(f_1,f_2)(x)=\int_{|x-y|^2+|x-z|^2>\varepsilon^2} K(x,y,z)f_1(y)f_2(z)dydz.
\end{equation}

\subsection{Dyadic cubes, adjacent systems and sparse operators}
The standard system of dyadic cubes in $\mathbb{R}^n$ is the collection $\mathcal{D}$,
\begin{equation}
  \mathcal{D}:=\{2^{-k}([0,1)^n + m): k\in\mathbb{Z}, m\in\mathbb{Z}^n\},
\end{equation}
consisting of simple half-open cubes of different length scales with sides parallel to the coordinate axes. These cubes satisfy the following properties:
\begin{enumerate}
  \item for any $Q\in\mathcal{D}$, the sidelength $\ell(Q)$ is of the form $2^k$, $k\in\mathbb{Z}$.\label{eq:dyadic1}
  \item $Q\cap R\in\{Q,R,\emptyset\}$, for any $Q,R\in\mathcal{D}$.\label{eq:dyadic2}
  \item the cubes of fixed sidelength $2^k$ form a partition of $\Rn$.\label{eq:dyadic3}
\end{enumerate}
Since given a ball $B(x,r)$, there does not always exist a cube $Q\in\mathcal{D}$ such that $B(x,r)\subset Q$ and $\ell(Q)\approx r$, a finite number of adjacent dyadic systems $\mathcal{D}^{u}$ can be used to overcome this problem. More precisely, these dyadic systems are the following
\begin{equation}
  \mathcal{D}^{u}:=\{2^{-k}([0,1)^u + m + (-1)^k\tfrac{1}{3}u): k\in\mathbb{Z}, m\in\mathbb{Z}^n\}, \hspace{1em} u\in\{0,1,2\}^n.
\end{equation}
The next two lemmas will be quite useful in the following. The first result can be found in \cite[Lemma 2.5]{HLP} in an stronger version whereas the second result is in \cite{HRT}.
\begin{Lemma}\label{lem:aprox}
  For any ball $B:=B(x,r)\subset\Rn$, there exists a cube $Q_B\in\mathcal{D}^u$ for some $u\in\{0,1,2\}^n$ such that $B\subset Q_B$ and $6r<\ell(Q_B)<12r$.
\end{Lemma}
Observe that, as a consequence of  \cite[Lemma 2.5]{HLP}, the collection $\mathcal{D}_0:=\cup_{u\in\{0,1,2\}^n}\mathcal{D}^u$ can be seen as a countable approximation of the collection of all balls in $\Rn$. This family satisfies \eqref{eq:dyadic1} and \eqref{eq:dyadic3} listed above, but it satisfies \eqref{eq:dyadic2} only in various weaker forms. We slightly abuse of the common terminology and say that $Q$ is a dyadic cube if $Q\in\mathcal{D}_0$.
\begin{Lemma}\label{lem:dyadicCoveringMod}
  If $Q_0\in\cup_{u\in\{0,1,2\}^n}\mathcal{D}^u$, then for any ball $B:=B(x,r)\subset Q_0$ there exists a cube $Q_B\in\cup_{u\in\{0,1,2\}^n}\mathcal{D}^u$ such that $B\subset Q_B\subseteq Q_0$ and $\ell(Q_B)\leq 12r$.
\end{Lemma}

\subsection{Multiple weights}

Along this section we recall some basic concepts related to some constants
involved in the multiple theory of weights.

\par First, let us define the central object in the multiple weight theory introduced in \cite{LOPTT}. Given $\vec f=(f_1,f_2)$, we define the multilinear maximal operator $\mathcal M$ by
$$\mathcal M(\vec f\,)(x)=\sup_{Q\ni x}\prod_{i=1}^m\frac{1}{|Q|}\int_Q|f_i(y_i)|dy_i,$$
where the supremum is taken over all cubes containing $x$.

Next, let us recall some useful definitions of the basic multiple weight constants that we are using throughout this paper. Consider numbers $1<p_1\ldots,p_m<\infty$ and $p$ such that $\frac{1}{p}=\frac{1}{p_1}+\ldots+\frac{1}{p_m}$ and denote $\vec{P}=(p_1,\ldots,p_m)$. Now define
\begin{equation}
[w,\vec\sigma]_{A_{\vec P}} := \sup_Q \langle w \rangle_Q \prod_{i=1}^m \langle \sigma_i \rangle_Q^{\frac{p}{p_i'}}.
\end{equation}
Notice that this definition is more general than that presented in \cite{LOPTT}, since when $\sigma_i=w_i^{1-p_i'}$, $i=1,\ldots,m$, and $w=\nu_{\vec w}$ we recover the $A_{\vec{P}}$ condition in \eqref{eq:AP_multiple} if $[w,\vec\sigma]_{A_{\vec P}}<\infty$.

In \cite{CD}, Chen and the first author introduced the following multilinear analogue of the $A_\infty$ constant, which was defined by Fujii in \cite{Fu} and later rediscovered by J.M Wilson \cite{Wil}. We say that $\overrightarrow w$ satisfies the $W_{\overrightarrow P}^\infty$ condition if
    \begin{equation}\label{eq:Fujii_constant}
      [\overrightarrow w]_{W_{\overrightarrow P}^\infty}=\sup_Q \Big(\int_Q\prod^m_{i=1}M(w_i \mathbf 1_Q)^{\frac{p}{p_i}}dx\Big)\Big(\int_Q\prod^m_{i=1}w_i^{\frac{p}{p_i}} dx\Big)^{-1}<\infty.
    \end{equation}
We can also define a more natural multilinear $A_{\infty}$ constant extending the classical Hruscev $A_{\infty}$ constant in \cite{Hru} as follows. We say that $\overrightarrow w$ satisfies the $H_{\vec{P}}^{\infty}$ condition if
\begin{equation}\label{eq:Hruscev_constant}
[\vec w]_{H_{\vec P}^\infty}:= \sup_{Q} \prod_{i=1}^m \langle w_i\rangle_Q^{\frac p{p_i}} \exp\Big( \dashint_Q \log w_i^{-1} \Big)^{\frac p{p_i}}.
\end{equation}

\section{Domination theorem for bilinear CZOs}\label{Sect:4}

In this section we will prove an extension of the domination theorem due to M. Lacey \cite{Lac} for bilinear Calder\'on--Zygmund operators following the scheme of proof in \cite{HRT} to track the precise constants.

\subsection{Some auxiliar operators and a related lemma}

Let $T$ be a bilinear Calder\'on--Zygmund operator with Dini-continuous kernel. For every cube $P\subset\Rn$, we defined the $P$-localized maximal truncation of $T$ as the operator

\begin{equation}
  T_{\sharp,P}(f_1,f_2)(x):= \sup_{0<\varepsilon<\delta<\tfrac{1}{2}dist(x,\partial P)} |T_{\varepsilon,\delta}(f_1,f_2)(x)| 1_P(x),
\end{equation}
where $T_{\varepsilon,\delta}$ is defined as follows
\begin{equation}
  T_{\varepsilon,\delta}(f_1,f_2)(x):= \iint_{\varepsilon^2<|x-y|^2+|x-z|^2<\delta^2} K(x,y,z)f_1(y)f_2(z)dzdy.
\end{equation}

We also need to define a truncated centered bilinear maximal function $\mathcal{M}_{\varepsilon,\delta}^c$ in the following way,
\begin{equation}
  \mathcal{M}_{\varepsilon,\delta}^c (f_1,f_2)(x):= \sup_{\varepsilon<r<\delta}\prod_{i=1}^2 \langle|f_i|\rangle_{B(x,r)}.
\end{equation}
We have the following relationship between the truncations $T_{\varepsilon,\delta}$ and $\mathcal{M}_{\varepsilon,\delta}^c$.
\begin{Lemma}\label{lem:truncatedMO}
  Suppose that $|x-x'|\leq \tfrac{1}{4}\varepsilon$. Then

  \begin{equation}\label{eq:approx}
    |T_{\varepsilon,\delta}(f_1,f_2)(x)-T_{\varepsilon,\delta}(f_1,f_2)(x')| \leq c_n (C_K+||\omega||_{\textup{Dini}})\mathcal{M}_{\varepsilon,2\delta}^c(f_1,f_2)(x).
  \end{equation}
\end{Lemma}

\begin{proof}

 First observe that

 \begin{align*}
   |T_{\varepsilon,\delta}&(f_1,f_2)(x)-T_{\varepsilon,\delta}(f_1,f_2)(x')|  \\
   &=\left| \iint_{\varepsilon^2<|x-y|^2+|x-z|^2<\delta^2} K(x,y,z)f_1(y)f_2(z) dzdy \right.   \\&-  \left.\iint_{\varepsilon^2<|x'-y|^2+|x'-z|^2<\delta^2} K(x',y,z)f_1(y)f_2(z)dzdy\right| \\
      &= \left| \iint_{\varepsilon^2<|x-y|^2+|x-z|^2<\delta^2} (K(x,y,z)-K(x',y,z))f_1(y)f_2(z)dzdy \right.   \\
       &+ \left(\iint_{\varepsilon^2<|x-y|^2+|x-z|^2<\delta^2} K(x',y,z)f_1(y)f_2(z)\right. dzdy \\ &\left.\left.-\iint_{\varepsilon^2<|x'-y|^2+|x'-z|^2<\delta^2} K(x',y,z)f_1(y)f_2(z) dzdy\right) \right|\\
       &:=|I + II|.
  \end{align*}
  For the first term, using the smoothness of the kernel and the properties of the modulus of continuity $\omega$, we get

\begin{align*}
 |I| &\lesssim  \iint_{\varepsilon^2<|x-y|^2+|x-z|^2<\delta^2}  \omega\left(\frac{ |x-x'| }{|x-y|+|x-z|}\right)\frac{|f_1(y)||f_2(z)|}{(|x-y|+|x-z|)^{2n}} dydz\\
 &\lesssim \sum_{k:\varepsilon^2\leq (2^k\varepsilon)^2<\delta^2}\iint_{(2^{k}\varepsilon)^2<|x-y|^2+|x-z|^2\le (2^{k+1}\varepsilon)^2}\hspace{-0.45cm} \omega\left(\frac{|x-x'|}{2^{k}\varepsilon}\right)\frac{|f_1(y)||f_2(z)|}{(2^k \varepsilon)^{2n}}dydz \\
 &\leq \sum_{k=0}^{\infty} \omega\left(\frac{|x-x'|}{2^k\varepsilon}\right)\iint_{B(x,2^{k+1}\varepsilon)} \frac{|f_1(y)||f_2(z)|}{(2^k \varepsilon)^{2n}}dydz \\
 &\leq c_n' \mathcal{M}^c_{\varepsilon,2\delta}(f_1,f_2)(x) \sum_{k=0}^{\infty} \int_{|x-x'|/2^k\varepsilon}^{|x-x'|/2^{k-1}\varepsilon} \omega(t)\frac{dt}{t} \\
 &\leq c_n' \mathcal{M}^c_{\varepsilon,2\delta}(f_1,f_2)(x)  \int_{0}^1 \omega(t)\frac{dt}{t}.
\end{align*}

For the second term, we make a similar decomposition as in \cite{HRT}, namely

\begin{align*}
  II = II_{\varepsilon}-II_{\delta},
\end{align*}
where
\begin{align*}
  II_r:&= \left( \iint_{|x-y|^2+|x-z|^2 > r^2} - \iint_{|x'-y|^2+|x'-z|^2 > r^2} K(x',y,z)f_1(y)f_2(z) dz dy\right) \\
  &= \iint_{|x-y|^2+|x-z|^2 > r^2\geq |x'-y|^2+|x'-z|^2 } K(x',y,z) f_1(y)f_2(z) dz dy \\
  &- \iint_{|x'-y|^2+|x'-z|^2 > r^2\geq |x-y|^2+|x-z|^2 } K(x',y,z) f_1(y)f_2(z) dz dy.
\end{align*}

Since $|x-x'|\leq \frac{\varepsilon}{4}\leq\frac{r}{4}$, for the first integral, $|x-y|+|x-z|\leq |x'-y|+|x'-z|+2|x-x'| \leq \sqrt{2}(|x'-y|^2+|x'-z|^2)^{1/2} +\frac{\varepsilon}{2}\leq \sqrt{2}r +\frac{r}{2}\leq 2r$, and $|x'-y|+|x'-z|\ge |x-y|+|x-z|-2|x-x'|\ge (|x-y|^2+|x-z|^2)^{\frac 12}-2|x-x'|\ge \frac r 2$, we have
\begin{align*}
&\Big| \iint_{|x-y|^2+|x-z|^2 > r^2\geq |x'-y|^2+|x'-z|^2 } K(x',y,z) f_1(y)f_2(z) dz dy\Big|\\
&\le \iint_{|x-y|+|x-z|\le 2r} \frac {C_K}{(r/2)^{2n}} |f_1(y)||f_2(z)|dydz\\
&\le c_n C_K \langle |f_1|\rangle_{B(x,2r)}\langle |f_2|\rangle_{B(x, 2r)}\le c_n C_K \mathcal M_{\varepsilon, 2\delta}^c (f_1,f_2)(x).
\end{align*}

For the second integral, we have $|x'-y|+|x'-z|\ge (|x'-y|^2+|x'-z|^2)^{\frac 12}\ge r$. Therefore,
\begin{align*}
&\Big|\iint_{|x'-y|^2+|x'-z|^2 > r^2\geq |x-y|^2+|x-z|^2 } K(x',y,z) f_1(y)f_2(z) dz dy\Big|\\
&\le \iint_{|x-y|^2+|x-z|^2\le r^2} \frac {C_K}{r^{2n}} |f_1(y)||f_2(z)|dydz\\
&\le c_n C_K \langle |f_1|\rangle_{B(x,r)}\langle |f_2|\rangle_{B(x, r)}\le c_n C_K \mathcal M_{\varepsilon, 2\delta}^c (f_1,f_2)(x).
\end{align*}
Consequently,
\[
II\le 4c_n C_K \mathcal M_{\varepsilon, 2\delta}^c (f_1,f_2)(x).
\]
which shows the desired result.
\end{proof}
The following result is an extension of the pointwise domination of the maximal truncation of $T$ by a sum of sparse operators in the bilinear setting.
\begin{Theorem}\label{Thm:domination}
Let $T$ be a bilinear Calder\'on--Zygmund operator with Dini continuous kernel. Then for any pair of compactly supported functions $f_1,f_2\in L^1(\Rn)$, there exist sparse collections $\mathcal{S}^{u}\subset\mathcal{D}^u$, $u=1,2,\ldots,3^n$, such that
\begin{equation}\label{ineq:pointwiseDomin}
  T_{\sharp}(f_1,f_2)(x) \leq c_n (||T||_{L^{q_1}\times L^{q_2}\to L^{q}} + C_K + ||\omega||_{\textup{Dini}}) \sum_{u=1}^{3^n} \mathcal{A}_{\mathcal{S}^u}(f_1,f_2)(x),
\end{equation}
for almost every $x\in\Rn$, where the constant $c_n$ depends only on the dimension and $||T||_{L^{q_1}\times L^{q_2}\to L^{q}}$ denotes the norm of the operator.
\end{Theorem}
The proof of the previous theorem follows exactly the same scheme of proof of \cite[Thm. 2.4]{HRT} with slight modifications. Notwithstanding, since the key ingredient for the proof of this theorem is essentially the next lemma, we are only going to give the details of its proof here for the sake of completeness.
\begin{Lemma}\label{Lem:recursion}
Let $f_1,f_2$ be integrable functions. Then, for every $Q_0\in\mathcal{D}_0$, there exists a collection $\mathcal{D}(Q_0)$ of dyadic cubes $Q\subset Q_0$ such that the following three conditions hold:
\begin{enumerate}
\item\label{recursion:prop1} $\sum_{Q\in \mathcal{D}(Q_0)}|Q| \leq \varepsilon_n |Q_0|$.
\item\label{recursion:prop2} if $Q'\subset Q$, and $Q',Q\in\mathcal{D}(Q_0)$, then $Q'=Q$.
\item\label{recursion:prop3} we have
\begin{equation}
  T_{\sharp,Q_0}(f_1,f_2) \leq C_T^0 \prod_{j=1}^2\langle|f_j|\rangle_{Q_0}\mathbf{1}_{Q_0} + \max_{Q\in\mathcal{D}(Q_0)} T_{\sharp,Q}(f_1,f_2),
\end{equation}
where $C_T^0:=c_n(||T||_{L^{q_1}\times L^{q_2}\to L^{q}} + C_K + ||\omega||_{\textup{Dini}})$ and $\varepsilon_n\searrow 0$ as $c_n\nearrow \infty$.
\end{enumerate}
\end{Lemma}

\begin{proof}
We want to prove that for any constant $C_T^0 > 0$ we can cover the set $E_0$,
\begin{equation*}
  E_0 := \left\{x\in Q_0: T_{\sharp,Q_0}(f_1,f_2)(x)>C_T^0 \prod_{j=1}^2 \langle|f_j|\rangle_{Q_0} \right\},
\end{equation*}
with countably many cubes $Q_i \in \mathcal{D}_0$ that satisfying conditions \eqref{recursion:prop2} and \eqref{recursion:prop3} and if the constant
$C^T_0$ is of the form $c_n\big(||T||_{L^{q_1}\times L^{q_2}\to L^{q}}+C_K+||\omega||_{\textup{Dini}}\big)$, then the cubes also satisfy condition \eqref{recursion:prop1}.

Let $x \in E_0$. Since the function $(\varepsilon,\delta) \mapsto T_{\varepsilon,\delta}(f_1,f_2)(x)$ is continuous, we can choose such radii $0<\sigma_x<\tau_x\leq \frac{1}{2} \cdot dist(x,\partial Q_0)$ that
\begin{equation*}
  \abs{T_{\sigma_x,\tau_x}(f_1,f_2)(x)}\geq C_T^0\prod_{j=1}^2\ave{\abs{f_j}}_{Q_0}
\end{equation*}
and
\begin{equation*}
  \abs{T_{\sigma,\tau}(f_1,f_2)(x)}\leq C_T^0\prod_{i=1}^2\ave{\abs{f_i}}_{Q_0}\qquad\text{if}\quad \sigma_x\leq\sigma\leq \tau\leq \frac{1}{2} \cdot dist(x,Q_0).
\end{equation*}
For simplicity, we drop the conditions $\varepsilon>0$ and $\delta\leq \frac{1}{2} \cdot dist(x,Q_0)$ from the notation. Now the maximality of $\sigma_x$ implies the following:
\begin{align*}
  T_{\sharp,Q_0}(f_1,f_2)(x) &=\sup_{\varepsilon\leq\delta}\abs{T_{\varepsilon,\delta}(f_1,f_2)(x)} \\
                     &=\sup_{\varepsilon\leq\delta\leq \sigma_x}\abs{T_{\varepsilon,\delta}(f_1,f_2)(x)} \\&\vee\sup_{\sigma_x\leq\varepsilon\leq\delta}\abs{T_{\varepsilon,\delta}(f_1,f_2)(x)}
                     \\  &\vee \sup_{\varepsilon\leq \sigma_x\leq \delta}\abs{T_{\varepsilon,\delta}(f_1,f_2)(x)} \\
                     & =:I\vee II\vee III,
\end{align*}
where
\begin{equation*}
  III=\sup_{\varepsilon\leq \sigma_x\leq\delta}\abs{T_{\varepsilon,\sigma_x}(f_1,f_2)(x)+T_{\sigma_x,\delta}(f_1,f_2)(x)}\leq I+II,
\end{equation*}
and $II\leq C_T^0\prod_{j=1}^2\ave{\abs{f_j}}_{Q_0}$ by definition. So altogether we find that
\begin{equation}\label{eq:localBound}
  T_{\sharp,Q_0}(f_1,f_2)(x)\leq \sup_{\varepsilon\leq\delta\leq \sigma_x}\abs{T_{\varepsilon,\delta}(f_1,f_2)(x)}+C^0_T\prod_{j=1}^2\ave{\abs{f_j}}_{Q_0}\hspace{1em}\forall x\in E_0,
\end{equation}
which is a preliminary version of the pointwise domination result we are proving. Now we can use Lemma \ref{lem:dyadicCoveringMod} to get from the preliminary version to the desired estimate. Since $B(x, 2\sigma_x) \subset Q_0$ for every $x \in E_0$, there exists a cube $Q_x \in \mathcal{D}_0$ such that $B(x, 2 \sigma_x) \subset Q_x \subset Q_0$ and $\ell(Q_x) \le 12 \cdot 2 \sigma_x$ for every $x \in E_0$. Let $(Q_i)_i$ be the sequence of such cubes $Q_x$ that are maximal with respect to inclusion, that is, for each $Q_i$ there does not exist $R \in \mathcal{D}_0$ such that $Q_i\subsetneq R \subseteq Q_0$. Then for every $x \in E_0$ we have
\begin{align*}
  T_{\sharp,Q_0} (f_1,f_2)(x) &\overset{\eqref{eq:localBound}}{\le} \sup_{0 < \varepsilon \le \delta \le \sigma_x} \left| T_{\varepsilon,\delta} (f_1,f_2)(x) \right| + C^0_T \prod_{j=1}^2\ave{|f_j|}_{Q_0} \\
                      &\le\sup_{0 < \varepsilon \le \delta \le \frac{1}{2} \cdot dist(x, \partial Q_x)} \left| T_{\varepsilon,\delta} (f_1,f_2)(x) \right| + C^0_T \prod_{j=1}^2\ave{|f_j|}_{Q_0} \\
                      &= T_{\sharp,Q_x} (f_1,f_2)(x) + C^0_T \prod_{j=1}^2\ave{|f_j|}_{Q_0} \\
                      &\leq \max_i T_{\sharp,Q_i}(f_1,f_2)(x) + C^0_T \prod_{j=1}^2\ave{|f_j|}_{Q_0}
\end{align*}
and for every $x \in Q_0 \setminus E_0$ we have $T_{\sharp,Q_0} (f_1,f_2)(x) \le C^0_T \prod_{j=1}^2\ave{|f_j|}_{Q_0}$ by definition. Thus, the cubes $Q_i$ satisfy Lacey's conditions (2) and (3) and to complete the proof, we only need to show that with a suitable choice of $C^0_T$ the cubes also satisfy property \eqref{recursion:prop1}.
Let us split the set $E_0$ into two parts:
\begin{equation*}
  E_1:=\{x\in E_0: \mathcal{M}_{\sigma_x,2\tau_x}(f_1,f_2)(x)\leq C^1_T\prod_{j=1}^2\ave{\abs{f_j}}_{Q_0}\},\qquad    E_2:=E_0\setminus E_1,
\end{equation*}
where $C^1_T$ is a constant whose value we will fix in the next step. Then, for $x\in E_1$ and $x'\in B(x,\frac14\sigma_x)$, we have
\begin{align*}
  \abs{T_{\sigma_x\tau_x}(f_1,f_2)(x')-T_{\sigma_x\tau_x}(f_1,f_2)(x)} &\overset{\ref{lem:truncatedMO}}{\le} c_n(C_K + ||\omega||_{\textup{Dini}})\mathcal{M}_{\sigma_x,2\tau_x}^c (f_1,f_2)(x) \\
                                                       &\le c_n(C_K + ||\omega||_{\textup{Dini}})C_T^1\prod_{j=1}^2\ave{\abs{f_j}}_{Q_0} \\
                                                       &=\frac12 C_T^0\prod_{j=1}^2\ave{\abs{f_j}}_{Q_0},
\end{align*}
provided that we choose
\begin{equation*}
  C_T^1:=\frac{C_T^0}{2c_n(C_K+||\omega||_{\textup{Dini}})}.
\end{equation*}
Then, since $x \in E_1 \subseteq E_0$, it follows that
\begin{equation*}\begin{split}
  T_{\sharp}(\mathbf 1_{Q_0}f_1,\mathbf 1_{Q_0}f_2)(x') &\ge \abs{T_{\sigma_x,\tau_x}(f_1,f_2)(x')} \\&\ge \abs{T_{\sigma_x,\tau_x}(f_1,f_2)(x)}-\frac12 C_T^0 \prod_{j=1}^2\ave{\abs{f_j}}_{Q_0} \\&> \frac12 C_T^0\prod_{j=1}^2\ave{\abs{f_j}}_{Q_0}
\end{split}\end{equation*}
for all $x'\in B(x,\frac14\sigma_x)$. In particular,
\begin{equation*}
\begin{split}
  \left| \bigcup_{x\in E_1}B(x,\tfrac14\sigma_x) \right|^2 &\leq \left| \{T_{\sharp}(\mathbf 1_{Q_0}f_1,\mathbf 1_{Q_0}f_2)>\tfrac12 C_T^0\prod_{j=1}^2\ave{\abs{f_j}}_{Q_0}\} \right|^2 \\
                                                         &\leq \frac{||T_{\sharp}||_{L^1\times L^1\to L^{1/2,\infty}}}{\tfrac12 C_T^0\prod_{j=1}^2\ave{\abs{f_j}}_{Q_0}}\prod_{i=1}^m ||\mathbf 1_{Q_0}f_i||_{L^1}
                                                         \\ &= \frac{2||T_{\sharp}||_{L^1\times L^1\to L^{1/2,\infty}}}{C_T^0}\abs{Q_0}^2
\end{split}
\end{equation*}
by the weak inequality of $T_\sharp$.

Let us then show that with this choice of $C^1_T$ and a suitable choice of $C_T^0$ the size of $E_2$ is controlled. Let $x \in E_2$. By definition, we can choose some $\rho_x\in [\sigma_x,2\tau_x]$ such that
\begin{equation*}
 \prod_{j=1}^2 \dashint_{B(x,\rho_x)}\abs{f_j(y_j)}dy_j >C_T^1\prod_{j=1}^2\ave{\abs{f_j}}_{Q_0}.
\end{equation*}
Since $\tau_x \le \frac{1}{2} \cdot dist(x,\partial Q_0)$, we know that $B(x,2\rho_x) \subset Q_0$. In particular,
\begin{equation*}
  \mathcal{M}(\mathbf 1_{Q_0}f_1,\mathbf 1_{Q_0}f_2)(x')>C^1_T\prod_{j=1}^2\ave{\abs{f_j}}_{Q_0}
\end{equation*}
for all $x'\in B(x,\rho_x)$, where $\mathcal{M}$ is the noncentered bilinear maximal operator
$$
\mathcal{M}(f_1,f_2)(x):=\sup_{B\ni x}\prod_{j=1}^2\dashint_B|f_j|dx.
$$
Thus
\begin{align*}
  \Big| \bigcup_{x\in E_2}&B(x,\frac14\sigma_x) \Big|^2 \\
  &\leq  \Big| \bigcup_{x\in E_2}B(x,\rho_x) \Big|^2 \leq\abs{\{\mathcal{M}(\mathbf 1_{Q_0}f_1,\mathbf 1_{Q_0}f_2)>C_T^1\prod_{j=1}^2\ave{\abs{f}}_{Q_0}\}}^2 \\
                                               &\leq\frac{c_n}{C^1_T\prod_{j=1}^2\ave{\abs{f_j}}_{Q_0}}\prod_{j=1}^2||\mathbf 1_{Q_0}f_j||_{L^1} =\frac{c_n'(C_K+||\omega||_{\textup{Dini}})}{C_T^0}\abs{Q_0}^2.
\end{align*}
by the weak inequality of the bilinear maximal operator.

Finally, let us combine all the previous calculations. For every maximal cube $Q_i$, let $x_i \in E_0$ be a point such that $Q_i = Q_{x_i}$. Then, since $\ell(Q_x) \le 12 \cdot 2\sigma_x$ for each $x \in E_0$, we have $|Q_{x_i}| \le c_n |B(x_i, \tfrac{1}{4} \sigma_{x_i})|$ for every $i$. In particular, since the cubes in the collection $\{ Q_{x_i} : Q_{x_i} \in \mathcal{D}^u\}$ are pairwise disjoint for a fixed $u \in \{0,1,2\}^n$ and $B(x_i, 2\sigma_{x_i})\subset Q_{x_i}$, $B(x_i, \frac 14\sigma_{x_i})$ are pairwise disjoint and therefore,
\begin{align*}
  \sum_i |Q_{x_i}| &= \sum_{u \in \{0,1,2\}^n}\sum_{i:Q_{x_i}\in\mathcal{D}^u} \left| Q_{x_i} \right| \\
     &  \leq c_n \sum_{u\in \{0,1,2\}^n} \sum_{i:Q_{x_i}\in\mathcal{D}^u} \left| B(x_i,\tfrac14\sigma_{x_i}) \right| \\
     &  = c_n \sum_{u \in \{0,1,2\}^n} \Big|  \bigcup_{i:Q_{x_i}\in\mathcal{D}^u } B(x_i,\tfrac14\sigma_{x_i}) \Big| \\
     &\le 3^n c_n \left( \Big| \bigcup_{x \in E_1} B(x,\tfrac{1}{4}\sigma_{x}) \Big| + \Big| \bigcup_{x \in E_2} B(x, \rho_x) \Big| \right) \\
     &\le c_n' \Big(\frac{||T_{\sharp}||_{L^{q_1}\times L^{q_2} \to L^{q}} + C_K + ||\omega||_{\textup{Dini}}}{C_T^0}\Big)^{\frac 12} |Q_0|.
\end{align*}
Hence, if
\begin{equation*}
  C^0_T = c_n(C_K+||\omega||_{\textup{Dini}}+||T_\sharp||_{L^{q_1}\times L^{q_2} \to L^{q}}),
\end{equation*}
then the cubes $Q_i$ satisfy property \eqref{recursion:prop1}.

\end{proof}

\section{Quantitative bounds for bilinear sparse operators}\label{Sect:5}
In this section we establish three different bounds for the family of bi-sublinear sparse operators $\mathcal A_{p_0,\gamma,\mathcal{S}}$. As a consequence of the domination theorem proved in the previous section, we will obtain the same bounds for bilinear Calder\'on-Zygmund operators or any other class of operators which can be controlled by this class of positive dyadic operators. For $\gamma>0$, $p_0\ge 1$, we define $\mathcal A_{p_0,\gamma,\mathcal{S}}$ as follows,
\[
\mathcal A_{p_0,\gamma,\mathcal{S}}(\vec{f})(x):=\left(\sum_{Q\in S} \left[\prod_{i=1}^2 \langle f_i \rangle_{Q,p_0}\right]^\gamma \mathbf 1_{Q}(x)\right)^{1/\gamma},
\]
where for any cube $Q$,
$$
\langle f \rangle_{Q,p_0} := \left(\frac{1}{|Q|} \int_Q |f(x)|^{p_0}dx \right)^{\frac{1}{p_0}}.
$$

Throughout this section we will use the following notation, $\tfrac{\vec{P}}{p_0}=(\tfrac{p_1}{p_0},\tfrac{p_2}{p_0})$.

Let us state our main results in this section. Our first bound is a mixed $A_{\vec{P}}$-$A_\infty$ estimate.

\begin{Theorem}\label{Thm:1}
        Let $\gamma>0$.
 Suppose that $p_0<p_1, p_2<\infty $ with $\frac{1}{p}=\frac{1}{p_1}+ \frac{1}{p_2}$. Let $w$ and $\vec\sigma$ be weights satisfying that $[w,\vec\sigma]_{A_{\vec P/p_0}}<\infty$ and $w, \sigma_i \in A_\infty$ for $i=1,2$. If  $\gamma \ge p_0$, then
\begin{align*}
\|\mathcal{A}_{p_0,\gamma,\mathcal S}&(\cdot\sigma_1, \cdot\sigma_2)\|_{L^{p_1}(\sigma_1)\times L^{p_2}(\sigma_2)\rightarrow L^p(w)} \\
&\lesssim [w,\vec \sigma]_{A_{\vec P/p_0}}^{\frac 1p}\Big( \prod_{i=1}^2 [\sigma_i]_{A_\infty}^{\frac 1{p_i}} +[w]_{A_\infty}^{(\frac 1\gamma-\frac 1p)_+}\sum_{j=1}^{2} \prod_{i\neq j}[\sigma_i]_{A_\infty}^{\frac 1{p_i}} \Big) ,
\end{align*}
where
\[
\left(\frac 1\gamma-\frac 1p\right)_+:=\max\left\{\frac 1\gamma-\frac 1p, 0\right\}.
\] If $\gamma < p_0$, then the above result still holds for all $p>\gamma $.
\end{Theorem}

Our second result is a mixed bound combining the $A_{\vec{P}}$ constant and a generalization of the  Fujii-Wilson $A_{\infty}$ constant to the bilinear setting which was introduced in \cite{CD}.

\begin{Theorem}\label{Thm:2}
 Let $\gamma>0$.
 Suppose that $p_0<p_1, p_2<\infty $ with $\frac{1}{p}=\frac{1}{p_1}+ \frac{1}{p_2}$ and set $q=p/\gamma$. Let $w$ and $\vec\sigma$ be weights satisfying that $[w,\vec\sigma]_{A_{\vec P/p_0}}<\infty$. If $\gamma \ge p_0$, then
\begin{equation}\begin{split}\label{eq:mixedbound}
\|\mathcal{A}_{p_0,\gamma,\mathcal S}(\cdot\sigma_1, \cdot\sigma_2)&\|_{L^{p_1}(\sigma_1)\times L^{p_2}(\sigma_2)\rightarrow L^p(w)} \\
&\le  [w,\vec{\sigma}]_{A_{\vec P/{p_0}}}^{1/p}([\vec \sigma]_{W_{\vec P}^\infty}^{1/p}+ \sum_{i=1}^2 [\vec \sigma^i]_{W_{\vec {P}^i}^\infty}^{1/{\gamma(\frac{p_i}\gamma)'}}),
\end{split}\end{equation}
where $[\vec \sigma^i]_{W_{\vec {P}^i}^\infty}=1$ if $p\le \gamma$ and otherwise,
\begin{align*}
[\vec \sigma^i]_{W_{\vec {P}^i}^\infty}=\sup_Q \Big(  \int_Q M &(\mathbf 1_Qw)^{\frac{(p_i/\gamma)'}{q'}} \prod_{j\neq i} M (\mathbf 1_Q\sigma_j)^{\frac{(p_i/\gamma)'}{p_j/\gamma}} dx\Big) \\ &\times\Big(\int_Q w^{\frac{(p_i/\gamma)'}{q'}}\prod_{j\neq i} \sigma_j^{\frac{(p_i/\gamma)'}{p_j/\gamma}}dx \Big)^{-1}.
\end{align*}
If  $\gamma < p_0$, then the above result still holds for all $p>\gamma $.
\end{Theorem}

Finally, we give a mixed bound combining the $A_{\vec{P}}$ constant and a generalization of the  Hrus\v{c}\v{e}v $A_{\infty}$ constant to the bilinear setting.
\begin{Theorem}\label{Thm:3}
Let $\gamma>0$.
 Suppose that $p_0<p_1, p_2<\infty $ with $\frac{1}{p}=\frac{1}{p_1}+ \frac{1}{p_2}$ and set $q=p/\gamma$. Let $w$ and $\vec\sigma$ be weights satisfying that $[w,\vec\sigma]_{A_{\vec P/p_0}}<\infty$. If $\gamma \ge p_0$, then
\begin{equation}\label{eq:mixedboundH}
\|\mathcal{A}_{p_0,\gamma,\mathcal S}(\cdot\sigma_1, \cdot\sigma_2)\|_{L^{p_1}(\sigma_1)\times L^{p_2}(\sigma_2)\rightarrow L^p(w)}\le [w,\vec\sigma]_{A_{\vec{P}/{p_0}}}^{\frac 1p} ([\vec \sigma]_{H_{\vec P}^\infty}^{1/p}+  \sum_{i=1}^2 [\vec \sigma^i]_{H_{\vec {P}^i}^\infty}^{1/{p_i'}}),
\end{equation}
where $[\vec \sigma^i]_{H_{\vec {P}^i}^\infty}=1$ if $p\le \gamma$ and otherwise,
\begin{equation}\begin{split}
[\vec \sigma^i]_{H_{\vec {P}^i}^\infty}=\sup_Q &\langle w \rangle_{Q}^{p_i'(\frac 1\gamma-\frac 1p)_+} \exp{\left(\dashint_Q \log w^{-1}\right)}^{p_i'(\frac 1\gamma-\frac 1p)_+} \\ &\times\prod_{j\neq i} \langle \sigma_i \rangle_{Q}^{p_i'/p_j} \exp{\left(\dashint_Q \log\sigma_i^{-1}\right)}^{p_i'/p_j}.
\end{split}\end{equation}
If  $\gamma < p_0$, then the above result still holds for all $p>\gamma $.
\end{Theorem}

Before proving Theorems~\ref{Thm:1}, \ref{Thm:2} and \ref{Thm:3}, we need the following two results. The first proposition can be found in \cite[Proposition 2.2]{cov2004}.

\begin{Proposition}\label{dyadicsum}
Let $1<s<\infty$, $\sigma$ be a positive Borel measure and
\[
\phi=\sum_{Q\in\mathcal D} \alpha_Q \mathbf 1_Q,\qquad \phi_Q=\sum_{Q'\subset Q}\alpha_{Q'} \mathbf 1_{Q'}.
\]
Then
\[
\|\phi\|_{L^s(\sigma)}\eqsim \Big( \sum_{Q\in \mathcal D} \alpha_Q (\langle\phi_Q\rangle_Q^\sigma)^{s-1}\sigma(Q) \Big)^{1/s}.
\]
\end{Proposition}

The following proposition follows the same spirit as that in \cite{HL} and it allows us to avoid the ``slicing'' argument. Namely, the separate consideration of families of cubes with the $A_{\vec P}$ characteristic ``frozen'' to a certain value $\langle w\rangle_Q\prod_{i=1}^2\langle \sigma_i\rangle_Q^{p/{p_i'}}\eqsim 2^k$.

By using Proposition~\ref{dyadicsum}, it is also possible to give an alternative proof of our main results by using the outer measure theory studied in \cite{DT,TTV}. Notice that here the stopping cubes method provides a more direct proof.

\begin{Proposition}\label{kolmogorov}
Let $\mathcal S$ be a sparse family and $0\le \gamma, \eta<1$ satisfying $\gamma+\eta<1$. Then
\begin{equation}\label{eq:kolmogorov}
\sum_{\substack{Q\in \mathcal S\\Q\subset R}}\langle u \rangle_Q^\gamma \langle v\rangle_Q^\eta |Q| \lesssim \langle u \rangle_R^\gamma \langle v\rangle_R^\eta |R|.
\end{equation}
\end{Proposition}

\begin{proof}
Indeed, set $1/r:=\gamma+\eta$, $1/s:=\gamma+(1-1/r)/2$ and $1/{s'}:=1-1/{s}$.  By sparseness and Kolmogorov's inequality, we have
\begin{align*}
\sum_{\substack{Q\in \mathcal S\\Q\subset R}}\langle u \rangle_Q^\gamma \langle v\rangle_Q^\eta |Q|&\le 2 \sum_{\substack{Q\in \mathcal S\\Q\subset R}}\langle u \rangle_Q^\gamma \langle v\rangle_Q^\eta |E_Q|\\
&\le 2 \int_{R} M(u\mathbf 1_R)^{\gamma} M(v\mathbf 1_R)^{\eta}dx\\
&\le 2 \Big(\int_R M(u\mathbf 1_R)^{s\gamma}  \Big)^{1/s}\Big(\int_R M(v\mathbf 1_R)^{s'\eta}  \Big)^{1/{s'}}\\
&\lesssim \langle u\rangle_R^\gamma |R|^{1/s} \langle v\rangle_R^{\eta}|R|^{1/{s'}}=\langle u \rangle_R^\gamma \langle v\rangle_R^\eta |R|.
\end{align*}
\end{proof}

Our first observation is that we can reduce the problem to study the case of $p_0=1$. Indeed,
consider the two weight norm inequality
\begin{equation}\label{eq:ep}
\|\mathcal A_{p_0,\gamma,\mathcal S }(f_1  ,f_2  )\|_{L^p(w)}\le \mathcal N(\vec P, p_0,\gamma, w, \vec \sigma) \|f_1\|_{L^{p_1}(w_1)} \|f_2\|_{L^{p_2}(w_2)},
\end{equation}
where we use $\mathcal N(\vec P, p_0,\gamma, w, \vec \sigma)$ to denote the best constant such that \eqref{eq:ep} holds.
Rewrite \eqref{eq:ep} as
\[
\|\mathcal A_{p_0,\gamma, \mathcal S}( f_1 ^{\tfrac{1}{p_0}} , f_2 ^{\tfrac{1}{p_0}} )\|_{L^p(w)}^{p_0}\le \mathcal N(\vec P, p_0,\gamma, w, \vec \sigma)^{p_0} \|f_1^{\tfrac{1}{p_0}} \|_{L^{p_1}(w_1)} ^{p_0} \|f_2^{\tfrac{1}{p_0}} \|_{L^{p_2}(w_2)}^{p_0},
\]
which is equivalent to the following
\[
\|\mathcal A_{1,\frac\gamma{p_0}, \mathcal S}( f_1   , f_2  )\|_{L^{p/{p_0}}(w)} \le \mathcal N(\vec P, p_0, \gamma,w, \vec \sigma)^{p_0} \|f_1  \|_{L^{p_1/{p_0}}(w_1)}  \|f_2  \|_{L^{p_2/{p_0}}(w_2)}.
\]
Therefore, if we denote by $\mathcal N(\vec P, \gamma,w,\sigma)$ the best constant for the case $p_0=1$, then the best constant for general $p_0$ would be $\mathcal N(\vec P/{ p_0}, \gamma/ {p_0}, w,\sigma)^{1/{p_0}}$. Therefore, it suffices to study the case of $p_0=1$.

Our second observation can be stated as follows, as it was done in \cite{HL, Li}.
\begin{Lemma}\label{lm:equivalent}
Suppose that $p> \gamma$.
Let $\mathcal N$ denote the best constant such that the following inequality holds
\begin{equation}\label{eq:p01}
\|\mathcal A_{1,\gamma,\mathcal S }(f_1\sigma_1  ,f_2 \sigma_2 )\|_{L^p(w)}\le \mathcal N  \|f_1\|_{L^{p_1}(\sigma_1)} \|f_2\|_{L^{p_2}(\sigma_2)}.
\end{equation}
Then \eqref{eq:p01} is equivalent to the following inequality with $\mathcal N'\simeq \mathcal N^\gamma$
\begin{equation}\label{eq:gamma1}
 \Big\|\Big(\sum_{Q\in \mathcal S} \langle f_1   \rangle_Q^{\sigma_1} \langle   f_2  \rangle_Q^{\sigma_2} \langle \sigma_1\rangle_Q^\gamma \langle\sigma_2\rangle_Q^\gamma \mathbf 1_Q\Big)^{\frac 1\gamma} \Big\|_{L^p(w)}^\gamma\le {\mathcal N'}  \|f_1\|_{L^{\frac{p_1}\gamma}(\sigma_1)} \|f_2\|_{L^{\frac{p_2}\gamma}(\sigma_2)}.
\end{equation}
\end{Lemma}
\begin{proof}
  On one hand, if \eqref{eq:gamma1} holds, we have
\begin{align*}
\|\mathcal A_{1,\gamma,\mathcal S }(f_1 \sigma_1 ,f_2 &\sigma_2 )\|_{L^p(w)} \\
&\le \Big\|\Big(\sum_{Q\in \mathcal S} \langle M_{\mathcal D}^{\sigma_1}(f_1 )^\gamma \rangle_Q^{\sigma_1} \langle M_{\mathcal D}^{\sigma_2}(f_2)^\gamma \rangle_Q^{\sigma_2} \langle \sigma_1\rangle_Q^\gamma \langle\sigma_2\rangle_Q^\gamma \mathbf 1_Q\Big)^{\frac 1\gamma} \Big\|_{L^p(w)}\\
&\lesssim\mathcal N  \| M_{\mathcal D}^{\sigma_1}(f_1 )^\gamma\|_{L^{p_1/\gamma}(\sigma_1)}^{1/\gamma}  \|M_{\mathcal D}^{\sigma_2}(f_2)^\gamma\|_{L^{p_2/\gamma}(\sigma_2)}^{1/\gamma}\\
&\le \mathcal N  \|f_1 \|_{L^{p_1}(\sigma_1)}   \|f_2\|_{L^{p_2}(\sigma_2)},
\end{align*}
where $M_{\mathcal{D}}^{\sigma}$ denotes the dyadic weighted maximal function, namely
\begin{equation}
  M_{\mathcal{D}}^{\sigma}(f)=\sup_{Q\in\mathcal{D}} \frac{1}{\sigma(Q)}\int_Q |f(x)|\sigma dx,
\end{equation}
which is bounded from $L^p(\sigma)$ into itself for every $p>1$.
On the other hand, if \eqref{eq:p01} holds, we have
\begin{align*}
&\Big\|\Big(\sum_{Q\in \mathcal S} \langle f_1   \rangle_Q^{\sigma_1} \langle   f_2  \rangle_Q^{\sigma_2} \langle \sigma_1\rangle_Q^\gamma \langle\sigma_2\rangle_Q^\gamma \mathbf 1_Q\Big)^{\frac 1\gamma} \Big\|_{L^p(w)}\\
&\le \Big\|\Big(\sum_{Q\in \mathcal S} (\langle M_{\gamma,\mathcal D}^{\sigma_1}(f_{1}^{1/\gamma})   \rangle_Q^{\sigma_1})^\gamma (\langle M_{\gamma,\mathcal D}^{\sigma_2}(f_{2}^{1/\gamma})   \rangle_Q^{\sigma_2})^\gamma \langle \sigma_1\rangle_Q^\gamma \langle\sigma_2\rangle_Q^\gamma \mathbf 1_Q\Big)^{\frac 1\gamma} \Big\|_{L^p(w)}\\
&\le \mathcal N \|  M_{\gamma,\mathcal D}^{\sigma_1}(f_{1}^{1/\gamma})    \|_{L^{p_1}(\sigma_1)} \|  M_{\gamma,\mathcal D}^{\sigma_2}(f_{2}^{1/\gamma})    \|_{L^{p_2}(\sigma_2)}\\
&\lesssim  \mathcal N \| f_{1}^{1/\gamma}    \|_{L^{p_1}(\sigma_1)} \|  f_{2}^{1/\gamma}    \|_{L^{p_2}(\sigma_2)},
\end{align*}
where $M_{\gamma,\mathcal{D}}^{\sigma}(f)=(M_{\mathcal{D}}^{\sigma}(f^{\gamma}))^{1/\gamma}$ and we have used in the last step that $p>\gamma$, which implies $p_1, p_2>\gamma$ and consequently, the boundedness of the maximal functions.
\end{proof}
Therefore, we further reduce the problem to study \eqref{eq:gamma1}. Finally, we
 give the following lemma, which is the key to prove Theorems \ref{Thm:1}, \ref{Thm:2} and \ref{Thm:3}.
\begin{Lemma}\label{lm:testing}
Let $\gamma>0$. Suppose that $1<p_1, p_2<\infty $ with $\frac{1}{p}=\frac{1}{p_1}+ \frac{1}{p_2}$. Let $w$ and $\vec\sigma$ be weights. Then for any sparse collection $\mathfrak S$,
\begin{equation}\label{eq:testing}
\Big\|\Big(\sum_{Q\in \mathfrak S} \langle \sigma_1\rangle_Q^\gamma \langle \sigma_2\rangle_Q^\gamma \mathbf 1_Q\Big)^{\frac 1\gamma}\Big\|_{L^p(w)}
\le [w,\vec \sigma]_{A_{\vec P}}^{\frac 1p}\Big( \sum_{Q\in \mathfrak S} \langle\sigma_1\rangle_Q^{\frac p{p_1}} \langle \sigma_2\rangle_Q^{\frac p{p_2}} |Q| \Big)^{1/p},
\end{equation}
and if $p>\gamma$, then there holds
\begin{equation}\begin{split}\label{eq:dualtesting}
\Big\| \sum_{Q\in \mathfrak S} \langle \sigma_1\rangle_Q^\gamma \langle \sigma_2 &\rangle_Q^{\gamma-1}\langle w\rangle_Q \mathbf 1_Q  \Big\|_{L^{(\frac{p_2}\gamma)'}(\sigma_2)} \\
&\le [w,\vec \sigma]_{A_{\vec P}}^{\frac \gamma p}\Big( \sum_{Q\in \mathfrak S} \langle \sigma_1\rangle_Q^{\frac {\gamma (\frac{p_2}{\gamma})'}{p_1}}\langle w\rangle_Q^{(\frac{p_2}\gamma)' (1-\frac \gamma p)} |Q|\Big)^{\frac 1{(\frac{p_2}\gamma)'}}
\end{split}\end{equation}
and
\begin{equation}\begin{split}\label{eq:dualtesting1}
\Big\| \sum_{Q\in \mathfrak S} \langle \sigma_1\rangle_Q^{\gamma-1} \langle \sigma_2 &\rangle_Q^{\gamma}\langle w\rangle_Q \mathbf 1_Q  \Big\|_{L^{(\frac{p_1}\gamma)'}(\sigma_1)}
\\
&\le [w,\vec \sigma]_{A_{\vec P}}^{\frac \gamma p}\Big( \sum_{Q\in \mathfrak S} \langle \sigma_2\rangle_Q^{\frac {\gamma (\frac{p_1}{\gamma})'}{p_2}}\langle w\rangle_Q^{(\frac{p_1}\gamma)' (1-\frac \gamma p)} |Q|\Big)^{\frac 1{(\frac{p_1}\gamma)'}}.
\end{split}\end{equation}
\end{Lemma}

\begin{proof}We start proving \eqref{eq:testing}. Following the spirit in \cite{HL}, observe that the right hand side of \eqref{eq:testing} is independent of $\gamma$. Therefore, it suffices to study the problem for small $\gamma$. More precisely, for fixed $\vec P$, we reduce the problem to study the case $\gamma< \min\{p,1\}$ with $(p/\gamma)'<\max\{p_1,p_2\}$. Without loss of generality, we may assume that $(p/\gamma)'<p_1=\max\{p_1,p_2\}$. Then it is easy to check that
\begin{equation}\label{eq:cond1}
0\le \gamma-\frac{\gamma p_1'}{p_2'}<1,\,\, 0\le 1-\frac{\gamma p_1'}p<1,
\end{equation}
and
\begin{equation}\label{eq:cond2}
 \gamma -\frac{\gamma p_1'}{p_2'}+ 1-\frac{\gamma p_1'}p<1.
\end{equation}
By Proposition \ref{dyadicsum}, we have
\begin{align*}
&\Big\| \Big(\sum_{Q\in \mathfrak S} \langle \sigma_1\rangle_Q^\gamma\langle\sigma_2\rangle_Q^\gamma \mathbf 1_Q\Big)^{1/\gamma}\Big\|_{L^p(w)}\\
&\eqsim \Big(\sum_{Q\in \mathfrak S} \lambda_Q \Big(\frac{1}{w(Q)}\sum_{Q'\subset Q}\langle \sigma_1\rangle_{Q'}^\gamma \langle\sigma_2\rangle_{Q'}^\gamma w(Q')\Big)^{\frac p\gamma-1}\Big)^{\frac 1p}\\
&\lesssim [w,\vec \sigma]_{A_{\vec P}}^{\frac {(p-\gamma)p_1'}{p^2}} \Big(\sum_{Q\in \mathfrak S} \lambda_Q \Big(\frac{1}{w(Q)}\sum_{Q'\subset Q}  \langle\sigma_2\rangle_{Q'}^{\gamma(1-\frac {p_1'}{p_2'})} \langle w\rangle_{Q'}^{1-\frac {p_1'\gamma}p} |Q'|\Big)^{\frac p\gamma-1}\Big)^{\frac 1p}\\
&\overset{\eqref{eq:kolmogorov}}{\lesssim} [w,\vec \sigma]_{A_{\vec P}}^{\frac {(p-\gamma)p_1'}{p^2}} \Big(\sum_{Q\in \mathfrak S} \lambda_Q \Big(\frac{1}{w(Q)} \langle\sigma_2\rangle_{Q }^{\gamma(1-\frac {p_1'}{p_2'})} \langle w\rangle_{Q }^{1-\frac {p_1'\gamma}p} |Q|\Big)^{\frac p\gamma-1}\Big)^{\frac 1p}\\
&= [w,\vec \sigma]_{A_{\vec P}}^{\frac {(p-\gamma)p_1'}{p^2}} \Big(\sum_{Q\in \mathfrak S} \langle \sigma_1\rangle_Q^{\gamma} \langle\sigma_2\rangle_Q^{\gamma+ (1-\frac {p_1'}{p_2'})(p-\gamma)}  \langle w\rangle_Q^{1-\frac{p_1'(p-\gamma)}p} |Q|\Big)^{1/p}\\
&\lesssim [w,\vec \sigma]_{A_{\vec P}}^{\frac {(p-\gamma)p_1'}{p^2}+\frac 1p-\frac {(p-\gamma)p_1'}{p^2}} \Big( \sum_{Q\in \mathfrak S} \langle\sigma_1\rangle_Q^{\frac p{p_1}} \langle \sigma_2\rangle_Q^{\frac p{p_2}} |Q| \Big)^{1/p}\\
&= [w,\vec \sigma]_{A_{\vec P}}^{\frac 1p } \Big( \sum_{Q\in \mathfrak S} \langle\sigma_1\rangle_Q^{\frac p{p_1}} \langle \sigma_2\rangle_Q^{\frac p{p_2}} |Q| \Big)^{1/p},
\end{align*}
where $\lambda_Q=\langle \sigma_1\rangle_Q^\gamma \langle\sigma_2\rangle_Q^\gamma w(Q)$. By symmetry, we only need to prove \eqref{eq:dualtesting}. Let us consider the case $(p/\gamma)'\ge \max\{p_1,p_2\}$ and $(p/\gamma)'< \max\{p_1,p_2\}$ separately. For the case $(p/\gamma)'<\max\{p_1,p_2\}$, without loss of generality, we may assume that $p_1>p_2$. Again, having into account \eqref{eq:cond1} and \eqref{eq:cond2} and using Proposition \ref{dyadicsum}, we obtain
 \begin{align*}
& \Big\| \sum_{Q\in \mathfrak S} \langle \sigma_1\rangle_Q^\gamma \langle \sigma_2\rangle_Q^{\gamma-1}\langle w\rangle_Q \mathbf 1_Q  \Big\|_{L^{(\frac{p_2}\gamma)'}(\sigma_2)}\\
&\simeq \Big( \sum_{Q\in \mathfrak S} \lambda_Q   \Big( \frac 1{\sigma_2(Q)}\sum_{Q'\subset Q}\langle \sigma_1\rangle_{Q'}^\gamma \langle \sigma_2\rangle_{Q'}^{\gamma } w(Q')   \Big)^{(\frac{p_2}\gamma)'-1}              \Big)^{\frac 1{(\frac{p_2}\gamma)'}}\\
&\le [w,\vec\sigma]_{A_{\vec P}}^{\frac {p_1'\gamma^2}{pp_2}}\Big( \sum_{Q\in \mathfrak S} \lambda_Q   \Big( \frac 1{\sigma_2(Q)}\sum_{Q'\subset Q}  \langle \sigma_2\rangle_{Q'}^{\gamma(1-\frac{p_1'}{p_2'}) }  \langle w\rangle_{Q'}^{  1-\frac{\gamma p_1'}{p}}   |Q'| \Big)^{(\frac{p_2}\gamma)'-1}              \Big)^{\frac 1{(\frac{p_2}\gamma)'}}\\
&\overset{\eqref{eq:kolmogorov}}{\lesssim} [w,\vec\sigma]_{A_{\vec P}}^{\frac {p_1'\gamma^2}{pp_2}}\Big( \sum_{Q\in \mathfrak S} \lambda_Q  \Big( \frac 1{\sigma_2(Q)}   \langle \sigma_2\rangle_{Q}^{\gamma(1-\frac{p_1'}{p_2'}) }  \langle w\rangle_{Q}^{  1-\frac{\gamma p_1'}{p}}  |Q|   \Big)^{(\frac{p_2}\gamma)'-1}              \Big)^{\frac 1{(\frac{p_2}\gamma)'}}\\
&=\hspace{-0.10cm} [w,\vec\sigma]_{A_{\vec P}}^{\frac {p_1'\gamma^2}{pp_2}}\hspace{-0.10cm} \Big( \hspace{-0.10cm} \sum_{Q\in \mathfrak S}\langle \sigma_1\rangle_Q^\gamma \langle \sigma_2\rangle_Q^{\gamma(\frac{p_2}{\gamma})'-(\frac{\gamma p_1'}{p_2'}+1)((\frac{p_2}\gamma)'-1) } \hspace{-0.15cm}  \langle w\rangle_Q^{(\frac {p_2}\gamma)'-\frac{\gamma p_1'}p((\frac {p_2}\gamma)'-1)}\hspace{-0.05cm} |Q|  \Big)^{\frac 1{(\frac{p_2}\gamma)'}}\\
&\le [w,\vec \sigma]_{A_{\vec P}}^{\frac \gamma p}\Big(\sum_{Q\in \mathfrak S} \langle \sigma_1\rangle_Q^{\frac {\gamma (\frac{p_2}{\gamma})'}{p_1}}\langle w\rangle_Q^{(\frac{p_2}\gamma)' (1-\frac \gamma p)} |Q|\Big)^{\frac 1{(\frac{p_2}\gamma)'}},
 \end{align*}
where $\lambda_Q=\langle \sigma_1\rangle_Q^\gamma \langle\sigma_2\rangle_Q^\gamma w(Q)$. It remains to consider when $(p/\gamma)'\ge \max\{p_1,p_2\}$. In this case,
 \[
   \gamma-\frac p{p_1'}\ge 0,\,\, \gamma-\frac p{p_2'}\ge 0.
 \]
 Moreover, since we are considering the case $p>\gamma$,
 \[
 \gamma-\frac p{p_1'}+\gamma-\frac p{p_2'}=2\gamma -2p+1<1.
 \]
 Applying Proposition \ref{dyadicsum} again, we have
  \begin{align*}
& \Big\| \sum_{Q\in \mathfrak S} \langle \sigma_1\rangle_Q^\gamma \langle \sigma_2\rangle_Q^{\gamma-1}\langle w\rangle_Q \mathbf 1_Q  \Big\|_{L^{(\frac{p_2}\gamma)'}(\sigma_2)}\\
&\simeq \Big( \sum_{Q\in \mathfrak S} \lambda_Q  \Big( \frac 1{\sigma_2(Q)}\sum_{Q'\subset Q}\langle \sigma_1\rangle_{Q'}^\gamma \langle \sigma_2\rangle_{Q'}^{\gamma } w(Q')   \Big)^{(\frac{p_2}\gamma)'-1}              \Big)^{\frac 1{(\frac{p_2}\gamma)'}}\\
&\le [w,\vec\sigma]_{A_{\vec P}}^{\frac {\gamma}{p_2}}\Big( \sum_{Q\in \mathfrak S} \lambda_Q   \Big( \frac 1{\sigma_2(Q)}\sum_{Q'\subset Q} \langle \sigma_1\rangle_{Q'}^{\gamma-\frac p{p_1'}} \langle \sigma_2\rangle_{Q'}^{\gamma-\frac p{p_2'} }    |Q'| \Big)^{(\frac{p_2}\gamma)'-1}              \Big)^{\frac 1{(\frac{p_2}\gamma)'}}\\
&\overset{\eqref{eq:kolmogorov}}{\lesssim} [w,\vec\sigma]_{A_{\vec P}}^{\frac { \gamma }{ p_2}}\Big( \sum_{Q\in \mathfrak S} \lambda_Q  \Big( \frac 1{\sigma_2(Q)}  \langle \sigma_1\rangle_{Q }^{\gamma-\frac p{p_1'}} \langle \sigma_2\rangle_{Q}^{\gamma-\frac p{p_2'} }    |Q|   \Big)^{(\frac{p_2}\gamma)'-1}              \Big)^{\frac 1{(\frac{p_2}\gamma)'}}\\
&= [w,\vec\sigma]_{A_{\vec P}}^{\frac { \gamma }{ p_2}} \Big( \sum_{Q\in \mathfrak S}\langle \sigma_1\rangle_Q^{\frac \gamma{p_2-\gamma}(p_2-\frac p{p_1'}) } \langle \sigma_2\rangle_Q^{\frac 1{p_2-\gamma}(p_2(1-\gamma)-\frac {p\gamma}{p_2'}) }   \langle w\rangle_Q  |Q|  \Big)^{\frac 1{(\frac{p_2}\gamma)'}}\\
&\le [w,\vec \sigma]_{A_{\vec P}}^{\frac \gamma p}\Big( \sum_{Q\in \mathfrak S} \langle \sigma_1\rangle_Q^{\frac {\gamma (\frac{p_2}{\gamma})'}{p_1}}\langle w\rangle_Q^{(\frac{p_2}\gamma)' (1-\frac \gamma p)} |Q|\Big)^{\frac 1{(\frac{p_2}\gamma)'}}.
 \end{align*}
where again $\lambda_Q=\langle \sigma_1\rangle_Q^\gamma \langle\sigma_2\rangle_Q^\gamma w(Q)$.
\end{proof}

Now we are ready to prove our main results.

\begin{proof}[Proof of Theorem~\ref{Thm:1}]
First we consider the case $p>\gamma$, and we denote $q=p/\gamma$. By Lemma \ref{lm:equivalent}, we have
\begin{align*}
& \Big\|\Big(\sum_{Q\in \mathcal S} \langle f_1   \rangle_Q^{\sigma_1} \langle   f_2  \rangle_Q^{\sigma_2} \langle \sigma_1\rangle_Q^\gamma \langle\sigma_2\rangle_Q^\gamma \mathbf 1_Q\Big)^{\frac 1\gamma} \Big\|_{L^p(w)}^\gamma \\&=\sup_{\|h\|_{L^{q'}(w)}=1}  \sum_{Q\in \mathcal S} \langle f_1   \rangle_Q^{\sigma_1} \langle   f_2  \rangle_Q^{\sigma_2} \langle \sigma_1\rangle_Q^\gamma \langle\sigma_2\rangle_Q^\gamma  \int_Q h \dw\\
&= \sup_{\|h\|_{L^{q'}(w)}=1} \sum_{Q\in \mathcal S}  \langle f_1\rangle_Q^{\sigma_1}   \langle f_2\rangle_Q^{\sigma_2}  \langle h\rangle_Q^w \langle \sigma_1\rangle_Q^\gamma \langle\sigma_2\rangle_Q^\gamma w(Q)
\end{align*}


For each $i=1,2$, let $\mathcal{F}_i$ be the stopping family starting at $Q_0$ and defined by the stopping condition

\begin{equation*}
  \ch_{\mathcal{F}_i}(\mathit{F_i}):=\{\mathit{F}'_i\in\mathcal{S} : \text{$F_i'\subset F_i$ maximal such that $\langle f_i \rangle^{\sigma_i}_{F'_i}>2 \langle f_i \rangle^{\sigma_i}_{F_i}$}\}.
\end{equation*}

Each collection $\mathcal{F}_i$ is $\sigma_i$-sparse, since
\begin{equation*}
  \sum_{F_i'\in \ch_{\mathcal{F}_i}(\mathit{F_i})} \sigma_i(F_i')\leq \frac{1}{2} \frac{\sum_{F_i'\in \ch_{\mathcal{F}_i}(\mathit{F_i})} \int_{F_i'} f d\sigma}{\int_{F_i} f d\sigma} \sigma_i(F_i) \leq \frac{1}{2}  \sigma_i(F_i).
\end{equation*}

The $\mathcal{F}_i$-stopping parent $\pi_{\mathcal{F}_i}(Q)$ of a cube $Q$ is defined by
$
\pi_{\mathcal{F}_i}(Q):=\{F_i \in\mathcal{F}_i : \text{$F_i$ minimal such that $F_i\supseteq Q$}\}.
$
By the stopping condition, for every cube $Q$ we have $\langle f_i \rangle^{\sigma_i}_{Q}\leq 2 \langle f_i \rangle^{\sigma_i}_{\pi_{\mathcal{F}_i}(Q)}$. Let $\mathcal{H}$ be the analogue stopping family associated with $h$ and the weight $w$, verifying the corresponding properties.

By rearranging the summation according to the stopping parents and removing the supremum, we obtain

\begin{align*}
&\sum_{Q\in \mathcal S}  \langle f_1\rangle_Q^{\sigma_1} \langle f_2\rangle_Q^{\sigma_2}  \langle h\rangle_Q^w \langle \sigma_1\rangle_Q^\gamma \langle\sigma_2\rangle_Q^\gamma w(Q)\\
&= \Bigg( \sum_{F_1\in{\mathcal F}_1} \sum_{\substack{F_2\in {\mathcal F}_2\\ F_2\subset F_1}} \sum_{\substack{H\in \mathcal H\\ H\subset F_2}}\sum_{\substack{Q \in\mathcal S\\ \pi(Q)=(F_1,F_2, H)}} + \sum_{F_2\in {\mathcal F}_2} \sum_{\substack{F_1\in {\mathcal F}_1\\ F_1\subset F_2}} \sum_{\substack{H\in \mathcal H\\ H\subset F}}\sum_{\substack{Q \in\mathcal S\\ \pi(Q)=(F_1,F_2, H)}}   \\
&+\sum_{F_1\in {\mathcal F}_1} \sum_{\substack{H\in \mathcal H\\ H\subset F_1}} \sum_{\substack{F_2\in {\mathcal F}_2\\ F_2\subset H}}\sum_{\substack{Q \in\mathcal S\\ \pi(Q)=(F_1,F_2, H)}} +\sum_{F_2\in{\mathcal F}_2} \sum_{\substack{H\in \mathcal H\\ H\subset F_2}} \sum_{\substack{F_1\in {\mathcal F}_1\\ F_1\subset H}}\sum_{\substack{Q \in\mathcal S\\ \pi(Q)=(F_1,F_2, H)}}\\
&+\sum_{H\in\mathcal H} \sum_{\substack{F_1\in {\mathcal F}_1\\ F_1\subset H}} \sum_{\substack{F_2\in {\mathcal F}_2\\ F_2\subset F_1}}\sum_{\substack{Q \in\mathcal S\\ \pi(Q)=(F_1,F_2, H)}} + \sum_{H\in\mathcal H} \sum_{\substack{F_2\in {\mathcal F}_2\\ F_2\subset H}} \sum_{\substack{F_1\in {\mathcal F}_1\\ F_1\subset F_2}}\sum_{\substack{Q \in\mathcal S\\ \pi(Q)=(F_1,F_2, H)}}\Bigg) \\ &\times \langle f_1\rangle_Q^{\sigma_1} \langle f_2\rangle_Q^{\sigma_2} \langle h\rangle_Q^w\lambda_Q\\
&:= I+I'+II+II'+III+III',
\end{align*}
where $\pi(Q)$ means that $\pi_{\mathcal{F}_i}(Q)=F_i$, for all $i=1,2$ and $\pi_{\mathcal{H}}(Q)=H$ and
\[
\lambda_Q:=  \langle \sigma_1\rangle_Q^\gamma \langle\sigma_2\rangle_Q^\gamma w(Q).
\]
By symmetry, it suffices to give an estimate for $I$. We have
\begin{align*}
I&\le \sum_{F_1\in{\mathcal F}_1} \sum_{\substack{F_2\in {\mathcal F}_2\\ F_2\subset F_1}} \sum_{\substack{H\in \mathcal H\\ H\subset F_2}}\sum_{\substack{Q \in\mathcal S\\ \pi(Q)=(F_1,F_2, H)}} \langle f_1\rangle_Q^{\sigma_1} \langle f_2\rangle_Q^{\sigma_2}  \langle h\rangle_Q^w \lambda_Q\\
&\le 8\sum_{F_1\in{\mathcal F}_1}  \langle f_1\rangle_{F_1}^{\sigma_1}  \sum_{\substack{F_2\in {\mathcal F}_2\\ F_2\subset F_1}}  \langle f_2\rangle_{F_2}^{\sigma_2} \sum_{\substack{H\in \mathcal H\\ H\subset F_2}} \langle h\rangle_H^w\sum_{\substack{Q \in\mathcal S\\ \pi(Q)=(F_1,F_2, H)}} \lambda_Q\\
&\lesssim  \sum_{F_1\in{\mathcal F}_1}  \langle f_1\rangle_{F_1}^{\sigma_1}  \sum_{\substack{F_2\in {\mathcal F}_2\\ F_2\subset F_1}}  \langle f_2\rangle_{F_2}^{\sigma_2}   \int \Big(\sup_{\substack{H'\in \mathcal H\\ \pi_{{\mathcal F}_2}(H')=F_2} } \langle h\rangle_{H'}^w \mathbf 1_{H'} \Big) \\ &\times\sum_{\substack{H\in \mathcal H\\ H\subset F_2}} \sum_{\substack{Q \in\mathcal S\\ \pi(Q)=(F_1,F_2, H)}} \frac{\lambda_Q}{w(Q)} \mathbf 1_Q\dw\\
&\le \sum_{F_1\in{\mathcal F}_1}  \langle f_1\rangle_{F_1}^{\sigma_1}  \sum_{\substack{{F_2}\in {\mathcal F}_2\\ \pi_{{\mathcal F}_1}(F_2)= F_1}}  \langle f_2\rangle_{F_2}^{\sigma_2}   \Big\|\sum_{\substack{H\in \mathcal H\\ \pi_{{\mathcal F}_2}(H)= F_2}} \sum_{\substack{Q \in\mathcal S\\ \pi(Q)=(F_1,F_2, H)}}  \frac{\lambda_Q}{w(Q)}\mathbf 1_Q \Big\|_{L^q(w)}\\
&\qquad\times \Big\|\sup_{\substack{H'\in \mathcal H\\ \pi_{{\mathcal F}_2}(H')=F_2} } \langle h\rangle_{H'}^w \mathbf 1_{H'}  \Big\|_{L^{q'}(w)}\\
&\leq \Big(\hspace{-0.2cm}\sum_{F_1\in{\mathcal F}_1}  \sum_{\substack{F_2\in {\mathcal F}_2\\ \pi_{{\mathcal F}_2}(F_2)= F_1}}  \hspace{-0.5cm}(\langle f_1\rangle_{F_1}^{\sigma_1} \langle f_2\rangle_{F_2}^{\sigma_2} )^q  \Big\|\hspace{-0.3cm}\sum_{\substack{H\in \mathcal H\\ \pi_{{\mathcal F}_2}(H)= F_2}} \sum_{\substack{Q \in\mathcal S\\ \pi(Q)=(F_1,F_2, H)}} \hspace{-0.7cm} \frac{\lambda_Q}{w(Q)} \mathbf 1_Q \Big\|_{L^q(w)}^q\Big)^{1/q}\\
&\qquad \times \Big(\sum_{F_1\in{\mathcal F}_1}  \sum_{\substack{F_2\in {\mathcal F}_2\\\pi_{{\mathcal F}_1}(F_2)= F_1}} \sum_{\substack{H'\in \mathcal H\\ \pi_{{\mathcal F}_2}(H')=F_2} } (\langle h\rangle_{H'}^w)^{q'} w(H') \Big )^{1/{q'}}\\
&\lesssim \hspace{-0.1cm}\Big(\sum_{F_1\in{\mathcal F}_1} \hspace{-0.2cm} \sum_{\substack{F_2\in {\mathcal F}_2\\ \pi_{{\mathcal F}_1}(F_2)= F_1}}  \hspace{-0.5cm}(\langle f_1\rangle_{F_1}^{\sigma_1} \langle f_2\rangle_{F_2}^{\sigma_2} )^q  \Big\|\hspace{-0.2cm}\sum_{\substack{H\in \mathcal H\\ \pi_{\mathcal F_2}(H)= F_2}} \sum_{\substack{Q \in\mathcal S\\ \pi(Q)=(F_1,F_2, H)}} \hspace{-0.7cm}\frac{\lambda_Q}{w(Q)} \mathbf 1_Q \Big\|_{L^q(w)}^q\Big)^{1/q}.
\end{align*}
By \eqref{eq:testing}, we have
\begin{align*}
  \Big\| \sum_{\substack{Q\in \mathcal S\\ \pi_{\mathcal F_2}(Q)= F_2}}\frac{\lambda_Q}{w(Q)}\mathbf 1_Q\Big\|_{L^q(w)}
 &\le [w,\vec \sigma]_{A_{\vec P}}^{\frac \gamma p}\Big( \sum_{\substack{Q\in \mathcal S\\ \pi_{\mathcal F_2}(Q)= F_2}} \langle\sigma_1\rangle_Q^{\frac p{p_1}} \langle \sigma_2\rangle_Q^{\frac p{p_2}} |Q| \Big)^{\gamma/p}.
\end{align*}


Therefore,
\begin{align*}
I&\le [w,\vec\sigma]_{A_{\vec{P}}}^{\frac \gamma p}  \Big(   \sum_{F_1\in{\mathcal F}_1}  \sum_{\substack{F_2\in {\mathcal F}_2\\ \pi_{{\mathcal F}_1}(F_2)= F_1}}  (\langle f_1\rangle_{F_1}^{\sigma_1} \langle f_2\rangle_{F_2}^{\sigma_2} )^q     \sum_{\substack{Q\in \mathcal S\\ \pi_{\mathcal F_2}(Q)= F_2}} \langle\sigma_1\rangle_Q^{\frac p{p_1}} \langle \sigma_2\rangle_Q^{\frac p{p_2}} |Q|    \Big)^{1/q}\\
&\le [w,\vec\sigma]_{A_{\vec{P}}}^{\frac \gamma p}  \Big(   \sum_{F_1\in{\mathcal F}_1}  \sum_{\substack{F_2\in {\mathcal F}_2\\ \pi_{{\mathcal F}_1}(F_2)= F_1}}  (\langle f_1\rangle_{F_1}^{\sigma_1} \langle f_2\rangle_{F_2}^{\sigma_2} )^q     \\
&\times\Big(\sum_{\substack{Q\in \mathcal S\\ \pi_{\mathcal F_2}(Q)= F_2}} \langle\sigma_1\rangle_Q |Q|\Big)^{\frac p{p_1}}  \Big(\sum_{\substack{Q\in \mathcal S\\ \pi_{\mathcal F_2}(Q)= F_2}} \langle\sigma_2\rangle_Q |Q|\Big)^{\frac p{p_2}}   \Big)^{1/q}\\
&\le [w,\vec\sigma]_{A_{\vec{P}}}^{\frac \gamma p}  \Big( \sum_{F_1\in{\mathcal F}_1} (\langle f_1\rangle_{F_1}^{\sigma_1})^{p_1/\gamma}\sum_{\substack{F_2\in {\mathcal F}_2\\ \pi_{{\mathcal F}_1}(F_2)= F_1}} \sum_{\substack{Q\in \mathcal S\\ \pi_{\mathcal F_2}(Q)= F_2}} \langle\sigma_1\rangle_Q |Q|  \Big)^{\frac \gamma{p_1}}\\
&\times \Big( \sum_{F_2\in{\mathcal F}_2} (\langle f_2\rangle_{F_2}^{\sigma_2})^{p_2/\gamma} \sum_{\substack{Q\in \mathcal S\\ \pi_{\mathcal F_2}(Q)= F_2}} \langle\sigma_2\rangle_Q |Q|  \Big)^{\frac \gamma{p_2}}\\
&\le [w,\vec\sigma]_{A_{\vec{P}}}^{\frac \gamma p}(\prod_{i=1}^2 [\sigma_i]_{A_\infty}^{\frac \gamma{p_i}}) \|    f_1 \|_{L^{p_1/\gamma}(\sigma_1)}\|f_2\|_{L^{p_2/\gamma}(\sigma_2)}.
\end{align*}
It remains to consider the case $p\le \gamma$. By Lemma \ref{lm:equivalent}, we have
\begin{align*}
 &\Big\|\Big(\sum_{Q\in \mathcal S}  \langle f_1\rangle_Q^{\sigma_1}   \langle f_2\rangle_Q^{\sigma_2}   \langle\sigma_1\rangle_Q^\gamma \langle\sigma_2\rangle_Q^\gamma \mathbf 1_Q\Big)^{\frac 1\gamma} \Big\|_{L^p(w)}^\gamma\\
&\lesssim  \Big\|\Big(\sum_{F_1\in \mathcal F_1}\langle f_1\rangle_{F_1}^{\sigma_1}\sum_{F_2\in \mathcal F_2}\langle f_2\rangle_{F_2}^{\sigma_2}\sum_{\substack{Q\in \mathcal S\\ \pi(Q)=(F_1, F_2)}}        \langle\sigma_1\rangle_Q^\gamma \langle\sigma_2\rangle_Q^\gamma \mathbf 1_Q\Big)^{\frac 1\gamma}\Big\|_{L^p(w)}^\gamma\\
&\le  \Big(\sum_{F_1\in \mathcal F_1}(\langle f_1\rangle_{F_1}^{\sigma_1})^q\sum_{F_2\in \mathcal F_2}(\langle f_2\rangle_{F_2}^{\sigma_2})^q \Big\|\sum_{\substack{Q\in \mathcal S\\ \pi(Q)=(F_1, F_2)}}        \langle\sigma_1\rangle_Q^\gamma \langle\sigma_2\rangle_Q^\gamma \mathbf 1_Q\Big\|_{L^q(w)}^q\Big)^{\frac 1q}\\
&\lesssim \Big( \sum_{F_1\in \mathcal F_1}(\langle f_1\rangle_{F_1}^{\sigma_1})^q\sum_{\substack{F_2\in \mathcal F_2\\ F_2\subset F_1}}(\langle f_2\rangle_{F_2}^{\sigma_2})^q \Big\|\sum_{\substack{Q\in \mathcal S\\ \pi(Q)=(F_1, F_2)}}        \langle\sigma_1\rangle_Q^\gamma \langle\sigma_2\rangle_Q^\gamma \mathbf 1_Q\Big\|_{L^q(w)}^q\Big)^{\frac 1q}\\
& +\Big(\sum_{F_2\in \mathcal F_2}(\langle f_2\rangle_{F_2}^{\sigma_2})^q \sum_{\substack{F_1\in \mathcal F_1\\F_1\subset F_2}}(\langle f_1\rangle_{F_1}^{\sigma_1})^q\Big\|\sum_{\substack{Q\in \mathcal S\\ \pi(Q)=(F_1, F_2)}}        \langle\sigma_1\rangle_Q^\gamma \langle\sigma_2\rangle_Q^\gamma \mathbf 1_Q\Big\|_{L^q(w)}^q\Big)^{\frac 1q}.
\end{align*}
Then by the previous arguments, the desired estimate follows. This completes the proof.
\end{proof}
The proof of Theorem~\ref{Thm:2} follows the same idea as the proof of the previous theorem.
\begin{proof}[Proof of Theorem \ref{Thm:2}]
We only check the estimate for $I$, since the other terms are similar. By \eqref{eq:testing}, we have
 \begin{align*}
   \Big\| \sum_{\substack{Q\in \mathcal S\\Q\subset F_2}} \langle \sigma_1\rangle_Q^\gamma\langle\sigma_2\rangle_Q^\gamma \mathbf 1_Q\Big\|_{L^q(w)}
 &\le [w,\vec \sigma]_{A_{\vec P}}^{\frac \gamma p}\Big( \sum_{\substack{Q\in \mathcal S\\Q\subset F_2}} \langle\sigma_1\rangle_Q^{\frac p{p_1}} \langle \sigma_2\rangle_Q^{\frac p{p_2}} |Q| \Big)^{\gamma/p}\\
 &\lesssim [w,\vec \sigma]_{A_{\vec P}}^{\frac \gamma p} \Big( \int_{F_2} \prod_{i=1}^2 M(\mathbf 1_{F_2}\sigma_i)^{p/{p_i}}dx               \Big)^{\gamma/p}\\
 &\le [w,\vec \sigma]_{A_{\vec P}}^{\frac \gamma p}[\vec \sigma]_{W_{\vec P}^\infty}^{\gamma/p} \Big( \int_{F_2} \prod_{i=1}^2 \sigma_i^{p/{p_i}}dx               \Big)^{\gamma/p}.
 \end{align*}
 Therefore,
  \begin{align*}
 I&\le [w,\vec\sigma]_{A_{\vec{P}}}^{\frac \gamma p} [\vec \sigma]_{W_{\vec P}^\infty}^{\gamma/p}\Big(   \sum_{F_1\in{\mathcal F}_1}  \sum_{\substack{F_2\in {\mathcal F}_2\\ \pi_{{\mathcal F}_1}(F_2)= F_1}}  (\langle f_1\rangle_{F_1}^{\sigma_1} \langle f_2\rangle_{F_2}^{\sigma_2} )^q      \int_{F_2} \prod_{i=1}^2 \sigma_i^{p/{p_i}}dx               \Big)^{1/q}\\
 &\le  \hspace{-0.05cm}[w,\vec\sigma]_{A_{\vec{P}}}^{\frac \gamma p} [\vec \sigma]_{W_{\vec P}^\infty}^{\gamma/p}\Big( \hspace{-0.10cm}         \int  \sum_{F_1\in{\mathcal F}_1} \hspace{-0.20cm}   (\langle f_1\rangle_{F_1}^{\sigma_1} )^q \mathbf 1_{F_1} \hspace{-0.20cm}   \sum_{\substack{F_2\in {\mathcal F}_2\\ \pi_{{\mathcal F}_1}(F_2)= F_1}} \hspace{-0.20cm}   (\langle f_2\rangle_{F_2}^{\sigma_2} )^q  \mathbf 1_{F_2}  \prod_{i=1}^2 \sigma_i^{p/{p_i}}dx               \Big)^{1/q}\\
 &\le [w,\vec\sigma]_{A_{\vec{P}}}^{\frac \gamma p}[\vec \sigma]_{W_{\vec P}^\infty}^{\gamma/p}\Big(          \int  M_{\mathcal D}^{\sigma_1} (f_1)^q  M_{\mathcal D}^{\sigma_2}(f_2)^q   \prod_{i=1}^2 \sigma_i^{p/{p_i}}dx               \Big)^{1/q}\\
 &\le [w,\vec\sigma]_{A_{\vec{P}}}^{\frac \gamma p} [\vec \sigma]_{W_{\vec P}^\infty}^{\gamma/p}  \| M_{\mathcal D}^{\sigma_1} (f_1)\|_{L^{p_1/\gamma}(\sigma_1)}\cdot \| M_{\mathcal D}^{\sigma_2} (f_2)\|_{L^{p_2/\gamma}(\sigma_2)}\\
 &\lesssim [w,\vec\sigma]_{A_{\vec{P}}}^{\frac \gamma p} [\vec \sigma]_{W_{\vec P}^\infty}^{\gamma /p} \| f_1\|_{L^{p_1/\gamma}(\sigma_1)}\cdot \|f_2\|_{L^{p_2/\gamma}(\sigma_2)}.
 \end{align*}

\end{proof}

Again, the proof of Theorem~\ref{Thm:3} also follows the same idea as the proof of the previous theorem.

\begin{proof}[Proof of Theorem~\ref{Thm:3}]
Likewise,  we only  study the estimate of $I$.
By \eqref{eq:testing} again, we have
\begin{align*}
& \Big\| \sum_{\substack{Q\in \mathcal S\\ \pi(Q)=F_2}} \langle \sigma_1\rangle_Q^\gamma\langle\sigma_2\rangle_Q^\gamma \mathbf 1_Q\Big\|_{L^q(w)}\\
 &\le [w,\vec \sigma]_{A_{\vec P}}^{\frac \gamma p}\Big( \sum_{\substack{Q\in \mathcal S\\ \pi(Q)=F_2}} \langle\sigma_1\rangle_Q^{\frac p{p_1}} \langle \sigma_2\rangle_Q^{\frac p{p_2}} |Q| \Big)^{\gamma/p}\\
&\le [w,\vec \sigma]_{A_{\vec P}}^{\frac \gamma p}[\vec \sigma]_{H_{\vec P}^\infty}^{\frac \gamma p}\Big( \sum_{\substack{Q\in \mathcal S\\ \pi(Q)=F_2}} \prod_{i=1}^2 \exp\Big( \dashint_Q \log \sigma_i  \Big)^{\frac p{p_i}} |Q| \Big)^{\gamma /p}\\
&\le [w,\vec \sigma]_{A_{\vec P}}^{\frac \gamma p}[\vec \sigma]_{H_{\vec P}^\infty}^{\frac \gamma p}\prod_{i=1}^2 \Big( \sum_{\substack{Q\in \mathcal S\\ \pi(Q)=F_2}} \exp\Big( \dashint_Q \log \sigma_i  \Big)  |Q| \Big)^{\gamma/{p_i}}\\
&\lesssim [w,\vec \sigma]_{A_{\vec P}}^{\frac \gamma p}[\vec \sigma]_{H_{\vec P}^\infty}^{\frac \gamma p}  \Big( \sum_{\substack{Q\in \mathcal S\\ \pi(Q)=F_2}} \exp\Big( \dashint_Q \log \sigma_1  \Big)  |Q| \Big)^{\gamma /{p_1}}\|M_0 (\mathbf 1_{F_2}\sigma_2)\|_{L^1}^{\frac \gamma {p_2}}\\
&\le [w,\vec \sigma]_{A_{\vec P}}^{\frac \gamma p}[\vec \sigma]_{H_{\vec P}^\infty}^{\frac \gamma p} \sigma_2(F)^{\frac \gamma {p_2}} \Big( \sum_{\substack{Q\in \mathcal S\\ \pi(Q)=F_2}} \exp\Big( \dashint_Q \log \sigma_1  \Big)  |Q| \Big)^{\gamma /{p_1}},
\end{align*}
where
\begin{equation}
M_0 (f) := \sup_Q \exp{\left( \dashint_Q \log{|f|}\right)}\mathbf 1_Q,
\end{equation}
is the logarithmic maximal function. Here we have used the fact that this maximal function is bounded from $L^p$ into itself for $p\in(0,\infty)$ with bound independent of the dimension in the dyadic case as proved in \cite[Lemma 2.1]{HP}.
Hence,
\begin{align*}
I&\le [w,\vec\sigma]_{A_{\vec{P}}}^{\frac \gamma p} [\vec \sigma]_{H_{\vec P}^\infty}^{\gamma /p}\Big(   \sum_{F_1\in{\mathcal F}_1} (\langle f_1\rangle_{F_1}^{\sigma_1})^q \sum_{\substack{F_2\in {\mathcal F}_2\\ \pi_{{\mathcal F}_1}(F_2)= F_1}}  ( \langle f_2\rangle_{F_2}^{\sigma_2} )^q \sigma_2(F)^{\frac p{p_2}}   \\
 &\quad\times \Big( \sum_{\substack{Q\in \mathcal S\\ \pi(Q)=F_2}} \exp\Big( \dashint_Q \log \sigma_1  \Big)  |Q| \Big)^{p/{p_1}}\Big)^{\frac \gamma p} \\
 &\le [w,\vec\sigma]_{A_{\vec{P}}}^{\frac \gamma p} [\vec \sigma]_{H_{\vec P}^\infty}^{\gamma/p}\Big(   \sum_{F_1\in{\mathcal F}_1} (\langle f_1\rangle_{F_1}^{\sigma_1})^q \Big(\sum_{\substack{F_2\in {\mathcal F}_2\\ \pi_{{\mathcal F}_1}(F_2)= F_1}}  ( \langle f_2\rangle_{F_2}^{\sigma_2} )^{p_2/\gamma} \sigma_2(F)\Big)^{\frac p{p_2}}\\
 &\quad\times \Big( \sum_{\substack{F_2\in {\mathcal F}_2\\ \pi_{{\mathcal F}_1}(F_2)= F_1}}\sum_{\substack{Q\in \mathcal S\\ \pi(Q)=F_2}} \exp\Big( \dashint_Q \log \sigma_1  \Big)  |Q| \Big)^{p/{p_1}}\Big)^{\frac \gamma p}\\
 &\le \hspace{-0.05cm}   [w,\vec\sigma]_{A_{\vec{P}}}^{\frac \gamma p} [\vec \sigma]_{H_{\vec P}^\infty}^{\gamma/p}\Big(   \sum_{F_1\in{\mathcal F}_1} (\langle f_1\rangle_{F_1}^{\sigma_1})^{p_1/\gamma} \Big(\hspace{-0.3cm}   \sum_{\substack{F_2\in {\mathcal F}_2\\ \pi_{{\mathcal F}_1}(F_2)= F_1}}\sum_{\substack{Q\in \mathcal S\\\pi(Q)=F_2}} \exp\Big( \dashint_Q \log \sigma_1  \Big)  |Q| \Big) \Big)^{\frac \gamma{p_1}}\\
 &\quad\times \Big(\sum_{F_1\in{\mathcal F}_1}\sum_{\substack{F_2\in {\mathcal F}_2\\ \pi_{{\mathcal F}_1}(F_2)= F_1}}  ( \langle f_2\rangle_{F_2}^{\sigma_2} )^{p_2/\gamma} \sigma_2(F)\Big)^{\frac \gamma{p_2}}\\
 &\le [w,\vec\sigma]_{A_{\vec{P}}}^{\frac \gamma p} [\vec \sigma]_{H_{\vec P}^\infty}^{\gamma/p} \|f_1\|_{L^{p_1/\gamma}(\sigma_1)}\|f_2\|_{L^{p_2/\gamma}(\sigma_2)}.
\end{align*}

\end{proof}

\section{Applications}\label{Sect:6}

\subsection{Mixed $A_p$-$A_\infty$ estimate for commutators of multilinear Calder\'on--Zygmund operators}\label{Sect:6.1}

Throughout this section, we will work with commutators of multilinear Calder\'on-- Zygmund operators with symbols in $BMO$. Recall that $BMO$ consists of all locally integrable functions $b$ with $||b||_{BMO}<\infty$, where
\begin{equation*}
\|b\|_{BMO}:= \sup_Q\frac 1{|Q|}\int_Q |b(y)-\langle b\rangle_Q| dy,
\end{equation*}
and the supremum in the above definition is taken over all cubes $Q\in \bbR^n$ with sides parallel to the axes.

Given a multilinear Calder\'on--Zygmund operator $T$ and $\vec{b}\in BMO^m$, we consider the following commutators with $\vec{b}$,
\[
[\vec b, T]=\sum_{i=1}^m [\vec b, T]_i,
\]
where
\[
[\vec b, T]_i(\vec f):= b_iT(\vec f)- T(f_1,\cdots, f_{i-1}, b_if_i, f_{i+1}, \cdots, f_m).
\]

Our aim in this section is to prove the following mixed estimate for commutators of multilinear Calder\'on-Zygmund operators following the same spirit as in \cite{CPP}.

\begin{Theorem}\label{Thm:comm}
Let $T$ be a multilinear Calder\'on-Zygmund operator and $\vec b\in BMO^m$. If we assume that $[w,\vec\sigma]_{A_{\vec P}}<\infty$, then
\begin{align*}
\|[\vec b, T&]\|_{L^{p_1}(w_1)\times\cdots\times L^{p_m}(w_m)\rightarrow L^p(w)} \\
&\le [w,\vec\sigma]_{A_{\vec P}}^{\frac 1p}(\prod_{i=1}^m [\sigma_i]_{A_{\infty}}^{\frac 1{p_i}}+[w]_{A_\infty}^{\frac 1{p'}}\sum_{j=1}^m\prod_{i\neq j}[\sigma_i]_{A_\infty}^{\frac 1{p_i}})\\&\times([w]_{A_\infty}+ \sum_{i=1}^m [\sigma_i]_{A_\infty})\left(\sum_{i=1}^m\|b_i\|_{BMO}\right),
\end{align*}
where $\sigma_i=w_i^{1-p_i'}$, $i=1,\ldots,m$.
\end{Theorem}

Before proving our main result in this section we need to recall some basic properties about $BMO$ functions and $A_{\infty}$ weights that we are going to use in the sequel.  Recall that a key property of $BMO$ functions is the celebrated John-Nirenberg inequality \cite{JN}.

\begin{Proposition}\cite[pp. 31-32]{J}
There are dimensional constants $0<\alpha_n <1 <\beta_n<\infty$ such that
\begin{equation}\label{eq:johnnirenberg}
\sup_Q \frac 1{|Q|}\int_Q \exp\Big( \frac{\alpha_n}{\|b\|_{\textup{BMO}}} |b(y)-\langle b\rangle_Q|\Big) dy\le \beta_n.
\end{equation}
In fact, we can take $\alpha_n=\frac 1{2^{n+2}}$.
\end{Proposition}

It is well-known that if $w\in A_\infty$, then $\log w\in BMO$. Using the John-Nirenberg inequality, Chung, Pereyra, and P\'erez \cite{CPP} proved the following bound.

\begin{Proposition}\label{prop:cpp}
Let $b \in BMO$ and let $0<\alpha_n <1 <\beta_n<\infty$  be the dimensional constants from \eqref{eq:johnnirenberg}. Then
\[
s\in \bbR,\,\, |s|\le \frac {\alpha_n}{\|b\|_{BMO}}\min\{1, \frac 1{p-1}\}\Rightarrow e^{sb}\in A_p\,\,\mbox{and}\,\, [e^{sb}]_{A_p}\le \beta_n^p.
\]
\end{Proposition}

In \cite{HP}, Hyt\"onen and P\'erez also showed the following bound for the Fujii-Wilson $A_{\infty}$ constant of a particular family of weights.

\begin{Proposition}\label{ainfty}
There are dimensional constants $\varepsilon_n$ and $c_n$ such that
\[
[e^{\re zb} w]_{A_\infty}\le c_n [w]_{A_\infty} \qquad \mbox{if}\,\, |z|\le \frac {\varepsilon_n}{\|b\|_{BMO} [w]_{A_\infty}}.
\]
\end{Proposition}

For our purpose, we need to show the following variation of the previous lemmas.

\begin{Lemma}\label{lm:prodweight}
Suppose that $[w,\vec \sigma]_{A_{\vec P}}<\infty$ and $w, \sigma_i\in A_\infty$, $i=1,2\cdots, m$. Then for any $1\le j\le m$,
\[
[we^{pb\re z}, \sigma_1,\cdots, \sigma_j e^{-p_j'b\re z}, \cdots, \sigma_m]_{A_{\vec {P}}}\le c_{n,\vec P} [w,\vec \sigma]_{A_{\vec P}},
\]
provided that
\[
|z|\le \frac {\alpha_n\min\{1, \frac{p_1'}p,\cdots, \frac{p_m'}p\}}{p(1+\max\{[w]_{A_\infty}, [\sigma_1]_{A_\infty},\cdots, [\sigma_m]_{A_\infty}\})\|b\|_{BMO}}.
\]
\end{Lemma}

To prove the previous lemma, we need to recall this sharp version of the reverse H\"older's inequality proved in \cite{HP}.

\begin{Proposition}\label{prop:reverseholder}
Let $w\in A_\infty$. Then for any $0\le r \le 1+ \frac {1}{c_n [w]_{A_\infty}}$, we have
\[
\Big(\frac 1{|Q|}\int_Q w(x)^r dx\Big)^{\frac 1r}\le 2 \frac 1{|Q|}\int_Q w(x) dx.
\]
\end{Proposition}

\begin{proof}[Proof of Lemma \ref{lm:prodweight}]
Set
\[
r= 1+\frac 1{c_n \max \{[w]_{A_\infty}, [\sigma_j]_{A_\infty}\}}.
\]
By definition of the $A_{\vec P}$ constant, H\"older's inequality and Proposition~\ref{prop:reverseholder}, we have
\begin{align*}
&[we^{pb\re z}, \sigma_1,\cdots, \sigma_j e^{-p_j'b\re z}, \cdots, \sigma_m]_{A_{\vec {P}}}\\
&=\sup_Q \langle we^{pb\re z}\rangle_Q \langle \sigma_j e^{-p_j'b\re z}\rangle_Q^{\frac p{p_j'}} \prod_{i\neq j} \langle \sigma_i\rangle_Q^{\frac p{p_i'}}\\
&\le \sup_Q \langle w^r\rangle_Q^{\frac 1r} \langle e^{pbr'\re z}\rangle_Q^{\frac 1{r'}} \langle \sigma_j^r\rangle_Q^{\frac p{rp_j'}}  \langle  e^{-p_j'br'\re z}\rangle_Q^{\frac p{r'p_j'}} \prod_{i\neq j} \langle \sigma_i\rangle_Q^{\frac p{p_i'}}\\
&\le 4 \sup_Q \langle w \rangle_Q  \langle e^{pbr'\re z}\rangle_Q^{\frac1{r'}} \langle \sigma_j \rangle_Q^{\frac p{p_j'}}  \langle  e^{-p_j'br'\re z}\rangle_Q^{\frac p{r'p_j'}} \prod_{i\neq j} \langle \sigma_i\rangle_Q^{\frac p{p_i'}}\\
&\le 4[w,\vec \sigma]_{A_{\vec P}} [e^{pbr'\re z}]_{A_{1+\frac p{p_j'}}}^{\frac 1{r'}}\\
&\le c_{n,\vec P} [w,\vec \sigma]_{A_{\vec P}},
\end{align*}
where Proposition \ref{prop:cpp} is used in the last step.
\end{proof}

Now we are ready to prove the main result in this section.

\begin{proof}[Proof of Theorem~\ref{Thm:comm}]
It suffices to study the boundedness of $[\vec b, T]_i$. Without loss of generality, we just consider the case $i=1$. Using the same trick as that in \cite[Thm. 3.1]{CPP}, for any complex number $z$, we define
\[
T^1_z (\vec f)= e^{zb} T(e^{-zb}f_1, f_2,\cdots, f_m).
\]
Then by using the Cauchy integral theorem, we get for ``nice" functions,
\[
[b, T]_1(\vec f)= \frac {d}{dz}T^1_z (\vec f)\Big|_{z=0}=\frac 1{2\pi i}\int_{|z|=\varepsilon} \frac {T^1_z(\vec f)}{z^2} dz,\,\, \varepsilon>0.
\]
Next, using Minkowski's inequality, for $p\ge 1$,
\begin{equation}\label{eq:minkowski}
\|[b, T]_1(\vec f)\|_{L^p(w)}\le \frac {1}{2\pi \varepsilon^2} \int_{|z|=\varepsilon} \|T^1_z(\vec f)\|_{L^p(w)} |dz|.
\end{equation}
Notice that
\begin{equation}\label{eq:relation}
\|T^1_z(\vec f)\|_{L^p(w)} = \|T(e^{-zb}f_1, f_2,\cdots, f_m)\|_{L^p(we^{pb\re z})}.
\end{equation}
Therefore, applying the boundedness properties for Calder\'on--Zygmund operators in Theorem~\ref{Thm:1} for weights $(we^{pb\re z},w_1 e^{p_1 b \re z}, w_2,\ldots,w_m)$ with $p_0=\gamma=1$, we get
\begin{equation}\begin{split}\label{eq:conjugation}
\|T(&e^{-zb}f_1, f_2,\cdots, f_m)\|_{L^p(we^{pb\re z})}
\lesssim [e^{pb\re z}w,e^{-p_1' b\re z}\sigma_1,\sigma_2,\ldots,\sigma_m]^{1/p}_{A_{\vec P}} \\
&\times \Big( [e^{-p_1'b\re z}\sigma_1]_{A_{\infty}}^{1/p_1}\prod_{i=2}^m [\sigma_i]_{A_{\infty}}^{1/p_i}+[e^{pb\re z}w]_{A_{\infty}}^{1/p'} \Big(\prod_{i=2}^m [\sigma_i]_{A_{\infty}}^{1/p_i} + \Big.\Big.\\
+& \Big.\Big. \sum_{i'=2}^m [\sigma_1 e^{-p_1'b\re{z}}]_{A_{\infty}}^{1/p_1}\prod_{\substack{i\neq i'\\ i>1}} [\sigma_i]_{A_{\infty}}^{1/p_i} \Big)\Big) ||f_1e^{-zb}||_{L^{p_1}(e^{bp_1\re z}w_1)}\prod_{i=2}^m ||f_i||_{L^{p_i}(w_i)}.
\end{split}\end{equation}
Combining \eqref{eq:minkowski}, \eqref{eq:relation} and \eqref{eq:conjugation} and using Proposition~\ref{ainfty} and Lemma~\ref{lm:prodweight}, we arrive at
\begin{equation}\begin{split}
\|[b,& T]_1(\vec f)\|_{L^p(w)} \\
&\leq \frac {1}{2\pi \varepsilon} [w,\vec\sigma]_{A_{\vec P}}^{1/p} \left( \prod_{i=1}^m [\sigma_i]_{A_{\infty}}^{1/p_i} + [w]_{A_{\infty}}^{1/p'}\sum_{i'=1}^m\prod_{i'\neq i}[\sigma_i]_{A_{\infty}}^{1/p_i}\right) \prod_{i=1}^m ||f_i||_{L^{p_i}(w_i)}.
\end{split}\end{equation}
Now taking
\[
\varepsilon= \frac {c_{n,\vec P}}{([w]_{A_\infty}+\sum_{i=1}^m[\sigma_i]_{A_\infty})\|b_1\|_{BMO}},
\]
where $c_{n,\vec P}$ is sufficiently small such that it satisfies the hypotheses in Proposition \ref{ainfty} and Lemma \ref{lm:prodweight}. Then,we obtain
\begin{equation*}\begin{split}
\|[b, T]_1(\vec f)\|_{L^p(w)}&\lesssim [w,\vec\sigma]_{A_{\vec P}}^{\frac 1p}(\prod_{i=1}^m [\sigma_i]_{A_{\infty}}^{\frac 1{p_i}}+[w]_{A_\infty}^{\frac 1{p'}}\sum_{j=1}^m\prod_{i\neq j}[\sigma_i]_{A_\infty}^{\frac 1{p_i}})\\&\times([w]_{A_\infty}+\sum_{i=1}^m[\sigma_i]_{A_\infty})||b_1||_{BMO}\prod_{i=1}^m ||f_i||_{L^{p_i}(w_i)}.
\end{split}\end{equation*}
The general result follows immediately combining the estimates for all the commutators in the different variables.
\end{proof}

\subsection{Mixed $A_p$-$A_\infty$ estimates for multilinear square functions and multilinear Fourier multipliers}\label{Sect:6.2}

The results obtained in Section~\ref{Sect:5} can be applied to different instances of operators which can be reduced to the simpler dyadic operators $\mathcal{A}_{p_0,\gamma,\mathcal{S}}$.

Firstly, observe that the mixed weighted bounds obtained in the main theorems in Section~\ref{Sect:5} can be extended to the case of multilinear square functions taking into account \cite[Prop. 4.2]{BuiHormozi} and choosing $p_0=1$ and $\gamma=2$.

These mixed bounds can also be extended to multilinear Fourier multipliers, which are a particular example of a general class of operators whose kernels satisfy weaker regularity conditions than the usual H\"older continuity. To obtain the corresponding mixed bounds, it is sufficient to consider the results in \cite{BCDH} together with the main theorems in Section~\ref{Sect:5} for $\gamma=1$. It is worth mentioning that these mixed bounds for Fourier multipliers seem to be new in the multilinear scenario.

\section{Appendix}\label{Sect:Appendix}

In this appendix we state and prove some well-known boundedness results for bilinear Calder\'on--Zygmund operators and their maximal truncations, which also hold in the multilinear setting. It is worth mentioning that the novelty of these results is not only that they are stated in a quantitative way that will be useful for our purposes, but also that some of these results are proved under weaker regularity conditions on the kernels than those results in the literature.


\begin{Lemma}\label{weak11}
Let $T$ be a bilinear Dini-continuous Calder\'on-Zygmund operator. Then $T$ is bounded from $L^1\times L^1$ to $L^{\frac 12, \infty}$ and
\begin{equation}\label{eq:weak11}
\|T\|_{L^{1}\times L^1\rightarrow L^{ 1/2,\infty}}\lesssim \|T\|_{L^{q_1}\times L^{q_2} \rightarrow L^q}+ \|\omega\|_{\textup{Dini}},
\end{equation}
where $|T\|_{L^{q_1}\times L^{q_2} \rightarrow L^q}$ denotes the norm of the operator as in its definition.
\end{Lemma}

This result was proved under the $Dini(\tfrac{1}{2})$ condition in \cite{MN}.  Observe that $Dini(1/2)$ condition is an stronger condition than Dini condition, which is also referred to as $Dini(1)$. In \cite{PT}, P\'erez and Torres studied the problem under the $BGHC$ condition. Namely, we say that a bilinear operator with kernel $K$ satisfies the bilinear geometric H\"ormander condition ($BGHC$) if  there exists a fixed constant $C$ such that and for any family of disjoint dyadic cubes $D_1$ and $D_2$,
\begin{equation}
  \int_{\Rn}\sup_{y\in Q}\int_{\mathbb{R}\setminus Q^*} |K(x,y,z)-K(x,y_Q,z)| dxdz \leq C,
\end{equation}
\begin{equation}
  \int_{\Rn}\sup_{z\in P}\int_{\mathbb{R}\setminus P^*} |K(x,y,z)-K(x,y,z_P)| dxdy \leq C,
\end{equation}
and
\begin{equation}\begin{split}
  \sum_{(P,Q)\in D_1\times D_2} \hspace{-0.5cm}|P| |Q| \sup_{(y,z)\in P\times Q}&\int_{\Rn \setminus(\cup_{R\in D_1})\cup(\cup_{S\in D_2)}} |K(x,y,z)-K(x,y_P,z_Q)|dx \\
  &\leq C (|\cup_{P\in D_1}P|+|\cup_{Q\in D_2}Q|).
\end{split}\end{equation}
Here $Q^*$ is the cube wit the same center as $Q$ and sidelength $10\sqrt{n} \ell(Q)$.This condition, which is actually stated here in an equivalent way, was shown to be weaker than the Dini condition in \cite[Prop. 2.3]{PT}). Thus, Lemma~\ref{weak11} follows immediately from the mentioned result. Here we give the proof with the precise constants.

\begin{proof}[Proof of Lemma~\ref{weak11}]

Suppose that $T$ is bounded from $L^{q_1}\times L^{q_2}$ to $L^q$, where $\frac 1{q_1}+\frac 1{q_2}=\frac 1q$. We shall dominate the bound $\|T\|_{L^{1}\times L^1\rightarrow L^{ 1/2,\infty}}$ by $\|T\|_{L^{q_1}\times L^{q_2} \rightarrow L^q}+ \|\omega\|_{\textup{Dini}}$. Indeed, fix $\lambda>0$ and consider without loss of generality functions $f_i\geq 0$, $i=1,2$. Let $\alpha_i>0$ be numbers to be determined later.  Apply the Calder\'on-Zygmund decomposition to $f_i$ at height $\alpha_i \lambda$, to obtain its good and bad parts $g_i$ and $b_i$, respectively, and families of cubes $\{Q_k^i\}_k$ with disjoint interiors such that $f_i=g_i+b_i$ and $b_i=\sum_k b_k^i$ verifying the properties in \cite[Thm. 4.3.1]{GrafC}.
\par Next, set $\Omega_i= \cup_k 4n Q_k^i$. We have
\begin{align*}
 \big| \{x: |T(f_1, f_2)(x)|>\lambda\} \big|  &\le |\Omega_1|+|\Omega_2| \\ &+\big| \{x\in  (\Omega_1\cup \Omega_2)^c: |T(g_1, g_2)(x)|>\frac \lambda 4\} \big|\\&+\big| \{x\in  (\Omega_1\cup \Omega_2)^c: |T(g_1, b_2)(x)|>\frac\lambda 4\} \big|\\
&+\big| \{x\in  (\Omega_1\cup \Omega_2)^c: |T(b_1, g_2)(x)|>\frac\lambda 4\} \big|\\
&+\big| \{x\in   (\Omega_1\cup \Omega_2)^c: |T(b_1, b_2)(x)|>\frac\lambda 4\} \big|.
\end{align*}
It is easy to see that
\[
 |\Omega_1|+|\Omega_2|\le C_n \Big(\frac 1{\alpha_1\lambda}\|f_1\|_{L^1}+\frac 1{\alpha_2\lambda}\|f_2\|_{L^1}\Big).
\]
For the third term, using Chebychev's inequality and the boundedness properties of $T$ and $g_i$, we have
\begin{align*}
 &\big| \{x\in (\Omega_1\cup \Omega_2)^c: |T(g_1, g_2)(x)|>\frac \lambda 4\} \big| \\&\le \frac {4^q} {\lambda^q} \|T(g_1, g_2)\|_{L^q}^q\\
&\le \frac {4^q} {\lambda^q} \|T\|_{L^{q_1}\times L^{q_2}\rightarrow L^q}^q \|g_1\|_{L^{q_1}}^q\|g_2\|_{L^{q_2}}^q\\
&\le \frac {4^q} {\lambda^q} C_{n,q,q_1,q_2} \|T\|_{L^{q_1}\times L^{q_2}\rightarrow L^q}^q (\alpha_1\lambda)^{q/{{q_1}'}} (\alpha_2\lambda)^{q/{{q_2}'}}\|f_1\|_{L^1}^{q/q_1}\|f_2\|_{L^1}^{q/q_2}.
\end{align*}
For the fourth term, if $c_k$ denotes the center of the cube $Q_k^2$, we have
\begin{align*}
&\big| \{x\in   (\Omega_1\cup \Omega_2)^c: |T(g_1, b^k_2)(x)|>\frac \lambda 4\} \big|\\
&\le \frac 4\lambda \int \Big|\sum_k\int \int_{Q_k^2} (K(x,y,z)-K(x,y,c_k)) g_1(y) b^k_2(z) dzdy\Big| dx\\
&\le  \frac 4\lambda\sum_k \int  \int \int_{Q_k^2} |K(x,y,z)-K(x,y,c_k)|\cdot |g_1(y)|\cdot |b^k_2(z)| dzdy dx\\
&\le  \frac 4\lambda\sum_k \int_{Q_k^2}\int  \int   \omega\Big( \frac{\sqrt n \ell (Q_k^2)}{2 (|x-y|+|x-z|)}\Big)\frac { |g_1(y)|\cdot |b^k_2(z)|}{(|x-y|+|x-z|)^{2n}}   dy dx dz\\
&\le    C_n\alpha_1 \sum_k \int_{Q_k^2}\int  \int   \omega\Big( \frac{\sqrt n \ell (Q_k^2)}{2 (| y|+|x-z|)}\Big)\frac {   |b^k_2(z)|}{(| y|+|x-z|)^{2n}}   dy dx dz\\
&\le  C_n\alpha_1 \sum_k \int_{Q_k^2}\int  \int  \omega\Big( \frac{\sqrt n \ell (Q_k^2)}{2 |x-z|}\Big)\frac {   |b^k_2(z)|}{(| y|+|x-z|)^{2n}}   dy dx dz\\
&\le  C_n\alpha_1 \sum_k \int_{Q_k^2}\int_{|x-z|\ge  n \ell(Q_k^2)}   \omega\Big( \frac{\sqrt n \ell (Q_k^2)}{2   |x-z| }\Big)\frac {   |b^k_2(z)|}{  |x-z| ^{ n}} dx dz\\
&\le C_n'\alpha_1 \|\omega\|_{\textup{Dini}} \|f_2\|_{L^1},
\end{align*}
where we have used the cancellation properties of $b^k_2$, the regularity condition on the third variable of $K$ (since $|z-c_k|<\tau \max{(|x-y|,|x-z|)}$ for $x\notin\Omega_1\cup\Omega_2$), the fact that $\omega$ is increasing, the Dini condition, $||g_1||_{L^{\infty}}\leq c_n\alpha_1\lambda$ and $\sum_k||b^k_2||_{L^1}\leq c_n ||f_2||_{L^1}$.

Since the estimate of the fifth term is symmetric to the previous estimate, it remains to estimate the last term. If we denote as $c_l$ and $c_k$ the center of the cubes $Q_l^1$ and $Q_k^2$, respectively, proceeding similarly as in the previous estimate, we obtain
\begin{align*}
&\big| \{x\in   (\Omega_1\cup \Omega_2)^c: |T(b_1, b_2)(x)|>\frac \lambda 4\} \big|\\
&\le \frac 4\lambda \int \Big|\sum_{k,l}\int_{Q_l^1} \int_{Q_k^2} (K(x,y,z)-K(x,y,c_k)) b^l_1(y) b^k_2(z) dzdy\Big| dx\\
&\le  \frac 4\lambda\sum_{k,l} \int_{(\Omega_1\cup\Omega_2)^c} \int_{Q_l^1} \int_{Q_k^2} |K(x,y,z)-K(x,y,c_k)| |b^l_1(y)| |b^k_2(z)| dxdy dx\\
&\le  \frac 4\lambda\sum_{k,l}  \int_{Q_k^2}\int_{Q_l^1} \int_{(\Omega_1\cup\Omega_2)^c}  \omega\Big( \frac{\sqrt n \ell (Q_k^2)}{2 (|x-y|+|x-z|)}\Big)\frac { |b^l_1(y)| |b^k_2(z)|dx  dy dz }{(|x-y|+|x-z|)^{2n}}  \\
&\le \frac{C_n}\lambda   \sum_{k,l}  \int_{Q_k^2}\int_{Q_l^1} \int_{(\Omega_1\cup\Omega_2)^c}  \omega\Big( \frac{\sqrt n \ell (Q_k^2)}{2 (|x-c_l|+|x-c_k|)}\Big)\frac { |b^l_1(y)||b^k_2(z)|dx  dy dz}{(|x-c_l|+|x-c_k|)^{2n}}  \\
&\le \hspace{-0.1cm} C_n  \sum_{k,l} |Q_l^1| |Q_k^2|\alpha_1 \alpha_2\lambda \hspace{-0.1cm}   \int_{(\Omega_1\cup\Omega_2)^c}  \hspace{-0.35cm}\omega\Big( \frac{\sqrt n (\ell (Q_k^2)+\ell(Q_l^1) )}{2 (|x-c_l|+|x-c_k|)}\Big)\frac {   dx}{(|x-c_l|+|x-c_k|)^{2n}}    \\
&\le  \hspace{-0.1cm}C_n'\sum_{k,l} \alpha_1 \alpha_2\lambda \int_{Q_k^2}\int_{Q_l^1}  \int_{(\Omega_1\cup\Omega_2)^c}  \hspace{-0.4cm} \omega\Big( \frac{\sqrt n (\ell (Q_k^2)+\ell(Q_l^1) )}{2 (|x-y|+|x-z|)}\Big)\frac {   dx dydz}{(|x-y|+|x-z|)^{2n}} \\
&= C_n'\sum_{k,l} \alpha_1 \alpha_2\lambda\Big( \int_{\ell(Q_k^2)\ge \ell(Q_l^1)}+ \int_{\ell(Q_l^1)\ge \ell(Q_k^2)}\Big)\\
&\le I+II.
\end{align*}
By symmetry, it suffices to estimate $I$. We have
\begin{align*}
I&\le  C_n'\sum_{k } \alpha_1 \alpha_2\lambda\int_{Q_k^2} \int_{(\Omega_1\cup\Omega_2)^c}\int_{\bbR^n}    \omega\Big( \frac{\sqrt n  \ell (Q_k^2)  )}{   |x-z| }\Big)\frac { dydxdz}{(|x-y|+|x-z|)^{2n}} \\
&= C_n'\sum_{k } \alpha_1 \alpha_2\lambda \int_{Q_k^2}\int_{(\Omega_1\cup\Omega_2)^c}    \omega\Big( \frac{\sqrt n  \ell (Q_k^2)  )}{   |x-z| }\Big)\frac {   1}{   |x-z| ^{ n}} dxdz\\
&\le C_n  \alpha_1 \|\omega\|_{\textup{Dini}} \|f_2\|_{L^1}.
\end{align*}
Combining the arguments above, we have
\begin{align*}
 \big| \{x: |T(&f_1, f_2)(x)|>\lambda\} \big| \\
 &\lesssim \frac 1{\alpha_1\lambda}\|f_1\|_{L^1}+\frac 1{\alpha_2\lambda}\|f_2\|_{L^1}\\
 &+\|T\|_{L^{q_1}\times L^{q_2}\rightarrow L^q}^q (\alpha_1 )^{q/{q_1'}} (\alpha_2 )^{q/{q_2'}}\lambda^{q-1}\|f_1\|_{L^1}^{q/q_1}\|f_2\|_{L^1}^{q/q_2}\\
 &+\alpha_1 \|\omega\|_{\textup{Dini}} \|f_2\|_{L^1}+\alpha_2 \|\omega\|_{\textup{Dini}} \|f_1\|_{L^1}
\end{align*}
Take
\begin{align*}
\alpha_1&= \lambda^{-\frac 12} \frac{\|f_1\|_{L^1}^{\frac 12}}{\|f_2\|_{L^1}^{\frac 12}}\frac 1{(\|T\|_{L^{q_1}\times L^{q_2}\rightarrow L^q}+\|\omega\|_{\textup{Dini}})^{\frac 12}}\\
\alpha_2&= \lambda^{-\frac 12} \frac{\|f_2\|_{L^1}^{\frac 12}}{\|f_1\|_{L^1}^{\frac 12}}\frac 1{(\|T\|_{L^{q_1}\times L^{q_2}\rightarrow L^q}+\|\omega\|_{\textup{Dini}})^{\frac 12}},
\end{align*}
we get
\[
\lambda \big| \{x: |T(f_1, f_2)(x)|>\lambda\} \big| ^2 \le (\|T\|_{L^{q_1}\times L^{q_2}\rightarrow L^q}+\|\omega\|_{\textup{Dini}}) \|f_1\|_{L^1}\|f_2\|_{L^1}.
\]
\end{proof}

%
We also need to show that the maximal truncated operator $T_\sharp$ is bounded from $L^1\times L^1$ to $L^{\frac 12, \infty}$. Therefore, we need to check first that Cotlar's inequality holds for this class of operators.
\begin{Theorem}\label{thm:Cotlar}
  Let $T$ be a bilinear Dini-continuous Calder\'on-Zygmund operator with kernel $K$. Then, for all $\eta\in(0,\tfrac{1}{2})$, there exists a constant $C$ such that
\begin{equation}\label{eq:cotlar}
T_\sharp (\vec f)\le  c_{\eta,n} (C_K + ||\omega||_{\textup{Dini}}+||T||_{L^{q_1}\times L^{q_2}\to L^q}) \mathcal{M}(\vec f)+ M_{\eta}(|T(\vec f)|).
\end{equation}
\end{Theorem}
In this proof we combine the strategies used in \cite[Thm 6.4]{MN} and \cite[Lemma 5.3]{HRT} to determine the precise constants involved in the inequality.

\begin{proof}[Proof of Theorem~\ref{thm:Cotlar}]
Let us begin defining the following maximal truncation
\[
\widetilde{T_\sharp}(f_1,f_2)(x)=\sup_{\varepsilon>0} \Big| \widetilde{T_\varepsilon}(f_1,f_2)(x) \Big|,
\]
where
\begin{align*}
\widetilde{T_\varepsilon}(f_1,f_2)(x)= \int_{\max\{|x-y|,|x-z|\}>\varepsilon} K(x,y,z) f_1(y)f_2(z)dydz.
\end{align*}
Since
\begin{equation}\label{eq:difference}
\sup_{\varepsilon>0}\left|\int_{\substack{\max\{|x-y|,|x-z|\}\le \varepsilon\\ |x-y|^2+|x-z|^2>\varepsilon^2}} K(x,y,z)f_1(y)f_2(z)dydz  \right|\lesssim C_K \mathcal M(f_1,f_2)(x),
\end{equation}
it suffices to show \eqref{eq:cotlar} with $T_\sharp$ replaced by $\widetilde{T_\sharp}$. Notice that we can write for $x'\in B(x,\varepsilon/2)$,
\begin{equation}\begin{split}\label{eq:CotlarAux}
\widetilde{T}_{\varepsilon}(f_1,f_2)(x)&= \int_{\max\{|x-y|,|x-z|\}>\varepsilon} (K(x,y,z) - K(x',y,z)) f_1(y)f_2(z) dydz \\
&+ T(f_1,f_2)(x')+T(f_1^0,f_2^0)(x'),
\end{split}\end{equation}
where $f_i^0=f_i \mathbf 1_{B(x,\varepsilon)}$. For the first term in \eqref{eq:CotlarAux}, using the regularity assumptions on the kernel, we get
\begin{align*}
&\Big|  \int_{\max\{|x-y|,|x-z|\}>\varepsilon} (K(x,y,z)-K(x',y,z)) f_1(y)f_2(z) dydz   \Big|\\
&\le    \int_{\max\{|x-y|,|x-z|\}>\varepsilon} \omega\Big(\frac{|x-x'|}{|x-y|+|x-z|}\Big)\frac{|f_1(y)| |f_2(z)| dydz}{(|x-y|+|x-z|)^{2n}}    \\
&= \sum_{k=0}^\infty \int_{2^k\varepsilon <\max\{|x-y|, |x-z|\}\le 2^{k+1}\varepsilon }   \omega\Big( \frac{|x-x'|}{2^{k}\varepsilon}\Big)\frac 1{(2^{k}\varepsilon)^{2n}}|f_1(y)||f_2(z)| dydz\\
&\lesssim \mathcal M(f_1,f_2)(x)  \sum_{k=0}^\infty   \omega\Big( \frac{|x-x'|}{2^{k}\varepsilon}\Big)\\
&\lesssim \mathcal M(f_1,f_2)(x) \sum_{k=0}^\infty \int_{2^{k-1}}^{2^k}\omega(\frac{|x-x'|}{\varepsilon t})\frac {dt}t\\
&=\mathcal M(f_1,f_2)(x) \sum_{k=0}^\infty \int_{\frac{|x-x'|}{2^k \varepsilon}}^{\frac{|x-x'|}{2^{k-1} \varepsilon}}\omega(u)\frac {du}u\\
&=\mathcal M(f_1,f_2)(x) \int_{0}^{\frac{2|x-x'|}{ \varepsilon}}\omega(u)\frac {du}u\\
&\le \|\omega\|_{\textup{Dini}} \mathcal M(f_1,f_2)(x),
\end{align*}
where the last step holds since $|x-x'|\le \varepsilon/2$.
Next, taking the $L^{\eta}$ average over $x'\in B(x,\varepsilon/2)$, we arrive at
\begin{align*}
|\widetilde{T_\varepsilon}(f_1,f_2)(x)|&\lesssim ||\omega||_{\textup{Dini}} \mathcal M(f_1,f_2)(x) + M_{\eta}(|T(f_1,f_2)|)(x)\\
&+ \left(\frac 1{|B(x, \varepsilon/2)|}\int_{B(x, \varepsilon/2)} |T(f_1^0,f_2^0)(x')|^\eta dx'\right)^{1/\eta}.
\end{align*}
For the last term, using Kolmogorov's inequality to relate the $L^{\eta}$ and $L^{1/2,\infty}$ norms and the boundedness of $T$ from $L^{1}\times L^1$ to $L^{1/2,\infty}$, we obtain for any $\eta\in(0,\tfrac{1}{2})$,
\begin{equation*}\begin{split}
\Big(\frac{1}{|B(x,\varepsilon/2)|}&\int_{B(x,\varepsilon/2)}|T(f_1^0,f_2^0)(x')|^\eta dx' \Big)^{1/\eta}
\\&= ||T(f_1^0,f_2^0)||_{L^{\eta}(B(x,\tfrac{\varepsilon}{2}),\frac{dx}{|B(x, \tfrac{\varepsilon}{2})|})}
\\&\leq C_{\eta}||T(f_1^0,f_2^0)||_{L^{1/2,\infty}(B(x, \tfrac{\varepsilon}{2}),\frac{dx}{|B(x, \tfrac{\varepsilon}{2})|})}
\\&\leq C_{\eta} ||T||_{L^1\times L^1\to L^{1/2,\infty}}\mathcal{M}(f_1,f_2)(x).
\end{split}\end{equation*}
Combining all the terms, we finally arrive at
\begin{equation*}\begin{split}
|\widetilde{T}_{\varepsilon}(f_1,f_2)(x)| &\leq c_{n}(||\omega||_{\textup{Dini}}+C_{\eta}||T||_{L^1\times L^1\to L^{1/2,\infty}}) \mathcal{M}(f_1,f_2)(x) \\&+ M_{\eta}(|T(f_1,f_2)|)(x),
\end{split}\end{equation*}
which taking into account \eqref{eq:difference} and \eqref{eq:weak11} leads to the desired result.
\end{proof}

As a corollary of the previous result follows the weak boundedness of the maximal truncation of $T$.

\begin{Corollary}
Let $T$ be a bilinear Calder\'on--Zygmund operator with Dini-con-tinuous kernel $K$. Then
\begin{equation}
||T_{\sharp}||_{L^1\times L^1 \to L^{1/2,\infty}} \lesssim (C_K + ||\omega||_{\textup{Dini}}+||T||_{L^{q_1}\times L^{q_2}\to L^q}).
\end{equation}
\end{Corollary}

\begin{proof}
 Fix $\eta\in(0,1/2)$ and use the previous result together with the weak boundedness of the multilinear maximal function and bilinear Calder\'on--Zygmund operators and the fact that $M_{\eta}\circ T:L^{1}\times L^1 \to L^{1/2,\infty}$. To prove the latter, notice that for the Hardy-Littlewood maximal function using Lemma~\ref{lem:dyadicCoveringMod}, we can write
\[M(f)\eqsim \sum_{u=1}^{3^n} M_u (f),\]
where
\[M_u(f):= \sup_{\substack{Q\ni x\\ Q\in \mathcal D^u}}\frac 1{|Q|}\int_Q |f(y)|dy.\]
Therefore,
\begin{align*}
\Big|\{x: M(|T(f_1,f_2)|^\eta)(x)^{\frac 1\eta}>\lambda\}\Big| &\le \sum_{u=1}^{3^n}\Big|\{x: M_u(|T(f_1,f_2)|^\eta)(x)^{\frac 1\eta}>\lambda/3^n\}\Big|.
\end{align*}
Denote
\[
E_u:= \{x\in \bbR^n: M_u(|T(f_1,f_2)|^\eta)(x)^{\frac 1\eta}>\lambda/3^n\}.
\]
We can find a collection of maximal dyadic cubes $\{Q_j\}_j$ such that $E_u=\cup_j Q_j$ and
\[
\frac 1{|Q_j|}\int_{Q_j}  |T(f_1,f_2)|^\eta > \lambda^\eta (3^n)^{-\eta},
\]
which means that
\[
|E_u|\le (3^n)^{\eta}\lambda^{-\eta} \int_{E_u} |T(f_1,f_2)|^\eta, \quad u=1,\ldots,3^n.
\]
Now using Kolmogorov's inequality and the fact that $T:L^1\times L^1 \to L^{1/2,\infty}$, and assuming that $\eta<1/2$, we get
\[
\int_{E_u} |T(f_1,f_2)|^\eta \lesssim \| T(f_1,f_2)\|_{L^{\frac 12, \infty}(E_u, \frac{dx}{|E_u|})}^\eta |E_u|\le \|f_1\|_1^\eta \|f_2\|_1^\eta |E_u|^{1-2\eta}
\]
Combining both estimates, it follows that
\[
|E_u|\le \lambda^{-\eta} (3^n)^{\eta} \|f_1\|_1^\eta \|f_2\|_1^\eta |E_u|^{1-2\eta},
\]
which is exactly,
\[
\lambda |E_u|^2\leq c_{n,\eta} \|f_1\|_1 \|f_2\|_1.
\]
\end{proof}

\section{Acknowledgements}
The authors would like to thank Prof. Tuomas Hyt\"onen for suggesting this problem and for carefully reading this manuscript as well as many helpful comments which have improved the quality of this paper.

\end{document}